\documentclass[11pt]{article}
\usepackage{cite,epic,eepic,euscript,verbatim,amsmath,amssymb,amsthm,afterpage,float,bm,stmaryrd,paralist,color,xcolor}
\usepackage{algorithm}
\usepackage{algorithmic}
\usepackage{graphicx}
\usepackage{subfigure,epstopdf}
\newtheorem{theorem}{Theorem}[section]
\newtheorem{lemma}{Lemma}[section]

\newtheorem{remark}{Remark}[section]

\numberwithin{equation}{section}

\def\XXint#1#2#3{{\setbox0=\hbox{$#1{#2#3}{\int}$}
\vcenter{\hbox{$#2#3$}}\kern-.51\wd0}}

\newcommand{\reff}[1]{{\rm (\ref{#1})}}

\newcommand{\R}{\mathbb{R}}            

\newcommand{\bn}{\mathbf n}            
\newcommand{\bnu}{{\bm \nu}}            
\newcommand{\ve}{\varepsilon}          

\newcommand{\bpsi}{\bm \psi}    
\newcommand{\bita}{\bm \eta} 
\newcommand{\bxi}{\bm \xi} 
 
\newcommand{\bu}{\bm u} 
 \newcommand{\bchi}{\bm \chi} 
\newcommand{\bv}{\bm v} 
\newcommand{\bd}{\bm d}
\newcommand{\bfb}{\mathbf b}
\newcommand{\bdelta}{\bm \delta}

\begin{document}


\title{Fast Iterative Method for Local Steric Poisson--Boltzmann Theories in Biomolecular Solvation}
\author{Wei Dou\thanks{
Department of Mathematics and Mathematical Center for Interdiscipline Research, Soochow University,  Suzhou, Jiangsu, China. Email: wdouwdou@stu.suda.edu.cn.}
\and
Minhong Chen\thanks{Department of Mathematics, Zhejiang Sci-Tech University, Hangzhou, Zhejiang, China. Corresponding Author, Email: mhchen@zstu.edu.cn.}
\and
Shenggao Zhou\thanks{School of Mathematical Sciences, MOE-LSC, and CMAI-Shanghai, Shanghai Jiao Tong University, Shanghai, China. Corresponding Author, Email: sgzhou@sjtu.edu.cn.}
}

\date{\today}

\maketitle

\begin{abstract}
This work proposes a fast iterative method for local steric Poisson--Boltzmann (PB) theories, in which the electrostatic potential is governed by the Poisson's equation and ionic concentrations satisfy equilibrium conditions. To present the method, we focus on a  local steric PB theory derived from a lattice-gas model, as an example. The advantages of the proposed method in efficiency are achieved by treating ionic concentrations as scalar implicit functions of the electrostatic potential, though such functions are only numerically achievable. 
The existence, uniqueness, boundness, and smoothness of such functions are rigorously established. A Newton iteration method with truncation is proposed to solve a nonlinear system discretized from the generalized PB equations. The existence and uniqueness of the solution to the discretized nonlinear system are established by showing that it is a unique minimizer of a constructed convex energy. Thanks to the boundness of ionic concentrations, truncation bounds for the potential are obtained by using the extremum principle. The truncation step in iterations is shown to be energy and error decreasing. To further speed-up computations, we propose a novel precomputing-interpolation strategy, which is applicable to other local steric PB theories and makes the proposed methods for solving steric PB theories as efficient as for solving the classical PB theory. Analysis on the Newton iteration method with truncation shows local quadratic convergence for the proposed numerical methods. Applications to realistic biomolecular solvation systems reveal that counterions with steric hindrance stratify in an order prescribed by the parameter of ionic valence-to-volume ratio. Finally, we remark that the proposed iterative methods for local steric PB theories can be readily incorporated in well-known classical PB solvers.
\end{abstract}

\noindent
{\bf Keywords}:
Steric Poisson--Boltzmann Theory; Newton Iteration Method with Truncation; Precomputing Speed-up Strategy; High Efficiency; Counterion Stratification.

{\allowdisplaybreaks


\section{Introduction}
\label{s:Introduction}
Electrostatic interactions between biomolecules, water molecules, and mobile ions are fundamental to the function and stability of solvated biomolecules, such as membranes and proteins~\cite{McCammon_PNAS09}. Ionic steric effects in electrostatic interactions have profound impact on the  dynamics of an underlying biomolecular system. For example, ionic steric effect plays a key role in the selection of ions that permeate through transmembrane channels~\cite{Hille_Book2001}.  Due to the long-range nature of the Coulomb force, electrostatic interactions are difficult to deal with, not only in continuum level but also in molecular simulations. As one of the most popular theories, the classical Poisson--Boltzmann (PB) theory has achieved great success in describing electrostatic interactions in biomolecular systems, electrochemical systems, etc.~\cite{DavisMcCammon_ChemRev90, Baker:COSB:2005}. Based on the mean-field approximation, the classical PB theory ignores direct ion-ion interaction details and assumes that ions distribute according to the mean-field electrostatic potential. With such approximations, the PB equation, which couples the Poisson's equation for the electrostatic potential and Boltzmann distributions for ionic concentrations, becomes a mathematically simple model that can be efficiently solved by various numerical methods~\cite{LuETC_CiCP08, LLP+:MBMB:2013}.  

The PB theory has been very successful in many applications. Nonetheless, the mean-field theory is known to neglect ionic steric effects and ion-ion correlations~\cite{Shklovskii_Rev02}. Without steric hindrance, the PB theory could predict unphysically high counterion concentration next to a charged surface~\cite{BAO_PRL97, ZhouWangLi_PRE11}. Recent years have seen a growing interest in incorporating the ionic steric effects within the framework of the PB theory~\cite{BAO_PRL97, BoschitschDanilov_SizeEffect_JCC12, Li_SIMA09,Li_Nonlinearity09,LuZhou_BiophysJ11,ZhouWangLi_PRE11}. The steric effect can be effectively incorporated in a variational approach by adding an excess chemical potential energy to the classical PB free energy~\cite{BazantReview_ACIS09}. One of the most popular local steric PB models is based on the statistical mechanics of ions and solvent molecules on lattices. In such a model, the excess chemical potential is described by the entropy of solvent molecules~\cite{Bikerman:PM:1942, BAO_PRL97, Li_Nonlinearity09, LuZhou_BiophysJ11, Horng_Entropy20}. The excess chemical potential can also represent hard-sphere interactions between ions using the Lennard-Jones (LJ) potential energy~\cite{HyonLiuBob_CMS10}. This leads to a nonlocal model, which can be further approximated using the Fourier analysis to obtain a computationally more tractable local model~\cite{HorngLinLiuBob_JPCB12, LinBob_CMS14, DingWangZhou_JCP19}.  Alternatively, the excess chemical potential can be determined by the Carnahan--Starling (CS) equation of state for hard-sphere liquids of a uniform size or the Boublik--Mansoori--Carnahan--Starling--Leland (BMCSL) equation of state for unequal sizes~\cite{BazantReview_ACIS09}. The excess chemical potential determined by the equation of state of CS or BMCSL can also derive local steric PB theories.  

Extensive efforts have been devoted to the development of numerical methods for the classical PB equation, ranging from finite difference methods~\cite{DavisMcCammon_JCC91, BruccoleriSharp_JCC97, IBR:CPC:1998, ZhouFeigWei_JCC2008, JWangRLuo_JCTC12}, finite element methods~\cite{BHW:JCC:2000, HBW:JCC:2000, ChenHolstXu_SINUM2000} to boundary element methods~\cite{LuChengHuangMcCammon_PNAS06,  GK:JCP:2013, GengJacob_CPC2013, LuChengHuangMcCammon_CPC13}. Such advances have led to the development of several successful software packages, such as APBS~\cite{BSJ+:PNASUSA:2001, APBS_Baker18}, DelPhi~\cite{DelPhiAlexov_CICP2013, LiJiaPengAlexov_BioinfM2017},  MIBPB~\cite{ZhouFeigWei_JCC2008, ChenGengWei_JCC2011}, and AFMPB~\cite{LuChengHuangMcCammon_CPC13}. Nonetheless, not many numerical methods have been developed for steric PB models, especially the models with nonuniform ionic sizes. One of the main obstacles to the development of an efficient solver for steric PB models with nonuniform ionic sizes is that the Boltzmann distributions in the classical PB theory are no longer available. An augmented Lagrange multiplier method has been proposed to minimize a free-energy functional subject to constraints for a steric PB model in the work~\cite{ZhouWangLi_PRE11}. Newton iterative relaxation finite element methods have been developed to solve steric PB models with nonuniform ionic sizes, which, after discretization, become a large nonlinear system coupling unknowns of both ionic concentrations and the electrostatic potential on computational grid points~\cite{Xie_IJNAM2017, LiXie_IJNAM2015, XieYingXie_JCC2017}.  The unknowns of each ionic concentrations on grid points were updated in each relaxation step in a coupled way~\cite{Xie_IJNAM2017}. Based on an observation made in the work~\cite{JiZhou_CMS19} that the unknowns of concentrations are spatially decoupled for each grid point with a given electrostatic potential, efficient and memory-saving iteration methods could be developed to solve steric PB models with nonuniform ionic sizes. 


In this work, we first briefly review several versions of local steric PB theories. We then focus on a widely used steric PB model based on the lattice-gas theory, and introduce the corresponding numerical discretization method for the Poisson's equation and equilibrium conditions for ionic concentrations. Since ionic concentrations are spatially decoupled for each grid point with a given electrostatic potential~\cite{JiZhou_CMS19}, we treat ionic concentrations as implicit scalar functions of the electrostatic potential, i.e., generalized Boltzmann distributions. The existence, uniqueness, boundness, and smoothness of such distributions are rigorously established for a given electrostatic potential. We develop a Newton iterative method with truncation for the nonlinear system resulting from discretization of the generalized PB equation, which couples the Poisson's equation with the generalized Boltzmann distributions. The existence and uniqueness of the solution to the discretized nonlinear system are proved by showing that the solution is a unique minimizer of a constructed convex energy functional. By the boundness of ionic concentrations,  we are able to establish upper and lower estimates on the electrostatic potential using the extremum principle. Such estimates are employed to truncate the solution in  Newton iterations. The truncation step is proved to be error and energy decreasing. Further analysis on the iteration method establishes that the method has local quadratic convergence rate. Numerical simulations are performed to demonstrate the efficiency and robustness of the proposed numerical methods. Applications to biomolecular solvation systems reveal that ions stratify next to the surface of biomolecules when the electrostatic potential is strong. Also, it is confirmed that the order of stratification is prescribed by the parameter of ionic valence-to-volume ratio. 

We remark that the proposed numerical method has salient features in efficiency and memory-saving. Several reasons account for such features: First, the iterations only involve the unknowns of electrostatic potential, rather than both potential and (possible multiple-species) ionic concentrations; Second, the derivative of concentrations with respect to the potential is used in iterations, making the super-linear convergence of the whole algorithm possible; Third, to further speed-up computations, we propose a precomputing-interpolation strategy, in which the generalized Boltzmann distributions and their derivatives are interpolated from precomputed and stored data. Such a strategy is directly applicable to other local steric PB theories, and makes the proposed algorithm for solving steric PB theories as efficient as for solving the classical PB theory.  Finally, we remark that our efficient iterative methods for the local steric PB theory can be readily incorporated to well-established PB solvers, such as the APBS~\cite{BSJ+:PNASUSA:2001, APBS_Baker18}, DelPhi~\cite{DelPhiAlexov_CICP2013, LiJiaPengAlexov_BioinfM2017},  MIBPB~\cite{ZhouFeigWei_JCC2008, ChenGengWei_JCC2011}, and AFMPB~\cite{LuChengHuangMcCammon_CPC13}.


The rest of this paper is organized as follows: In Section~\ref{s:Theory}, we introduce local steric PB theories. In Section~\ref{s:NumMethods}, we propose numerical methods, and present some analysis on the model and numerical methods. In Section~\ref{s:NumRes},  we report some numerical simulations to demonstrate the effectiveness of the numerical methods.  Finally, we draw conclusions in Section~\ref{s:Conclusions}.

\section{Local Steric Poisson--Boltzmann Theories}
\label{s:Theory}
\begin{figure}
\begin{center}
\includegraphics[scale=0.6]{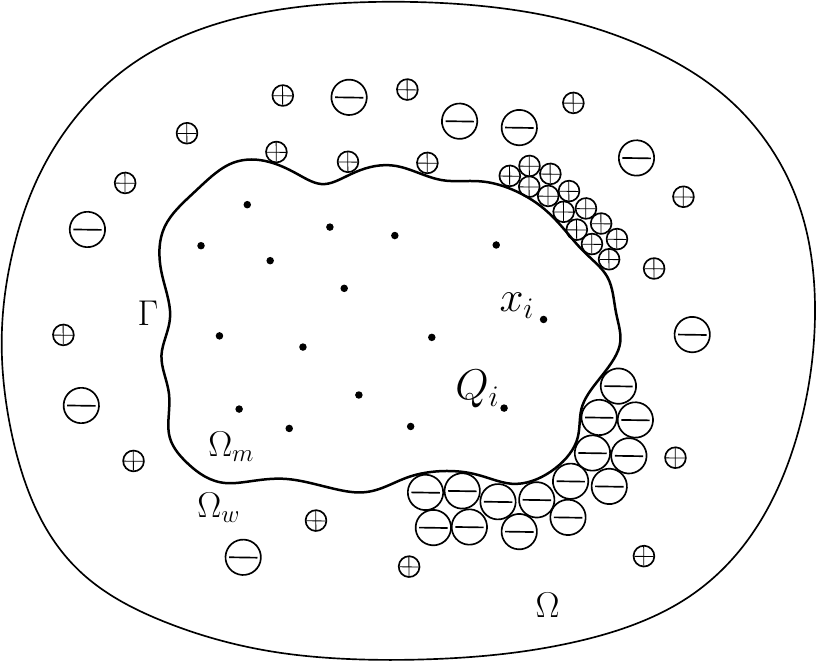}
\caption{A schematic view of a charged biomolecule solvated in an ionic solution with $M $ $(M \ge 1)$ ionic species. The system occupies a region $\Omega$, which is divided into a solute region $\Omega_{\rm m}$ and a solvent region $\Omega_{\rm w}$ by a solute-solvent interface $\Gamma$. The solute contains atoms, located at $x_i$, carry partial charges $Q_i$ ($1\leq i \leq N$).  The ions are shown individually to highlight the steric effects of ions in electric double layers next to charged biomolecules. 
}
\label{f:geometry}
\end{center}
\end{figure}
Consider the solvation of charged biomolecules in an ionic solution, in which the aqueous solvent is treated implicitly as a continuum and the distribution of ions is described by ionic concentrations. As shown in the schematic view of Fig.~\ref{f:geometry}, the solvation system occupies a region $\Omega$, which is an open and connected subset of $\R^3$ with a smooth boundary $\partial \Omega$. A solute-solvent interface $\Gamma$ separates the solute region, $\Omega_{\rm m}$, from the solvent region, $\Omega_{\rm w}$, and $\bn$ denotes a unit normal vector defined on $\Gamma$. Assume that the solute consists of $N$ atoms located at $x_i$, carrying fixed partial charges $Q_i$ ($1\leq i \leq N$). Suppose that there are $M $ $(M \ge 1)$ ionic species in the ionic solution. Deonte by $z_i$ and $c_i = c_i(x)$ the valence and local ionic concentration of the $i$th ($1 \le i \le M$) ionic species at position $x \in \Omega$, respectively. Let $c = (c_1, \dots, c_M).$ The volume associated to an ion of the $i$th species is denoted by $v_i$ $(1 \le i \le M)$.


The electrostatic potential $\psi$ of the charged system is governed by a boundary-value problem (BVP) of the Poisson's equation
\begin{equation}\label{BVP}
\left\{
\begin{aligned}
& -\nabla \cdot \ve_0 \ve_r \nabla \psi =  \rho^f \qquad & & \mbox{in } \Omega_{\rm m}, \\
& -\nabla \cdot \ve_0 \ve_r \nabla \psi =  \sum_{i=1}^M z_i e c_i \qquad & & \mbox{in } \Omega_{\rm w}, \\
& [[\psi]] =0~ \mbox{and } [[\ve_r \frac{\partial \psi}{\partial \bn}]]= 0 & & \mbox{on } \Gamma,\\
& \psi = \psi^\infty  \qquad & & \mbox{on } \partial \Omega,
\end{aligned}
\right.
\end{equation}
where $\ve_0$ is the vacuum permittivity,  the dielectric coefficient $\ve_r = \ve_r (x)$ is $\ve_{\rm w}$ if $x \in \Omega_{\rm w}$ and is $\ve_{\rm m}$ if $x \in \Omega_{\rm m}$,  $\rho^f=\sum_{i=1}^N Q_i \delta_{x_i}$ is the fixed charge distribution representing point charges carried by solute atoms, $e$ is the elementary charge, $[[\zeta]]=\zeta|_{\Omega_{\rm w}}-\zeta|_{\Omega_{\rm m}}$ denotes the jump across $\Gamma$ of a function $\zeta$ from $\Omega_{\rm m}$ to $\Omega_{\rm w}$, 
and $\psi^\infty : \partial \Omega \to \R$ is a given bounded and smooth function representing the electrostatic potential on the boundary. Define the total charge density $\rho$ by
$$\rho=\chi^w\sum_{i=1}^M z_i e c_i + \rho^f,$$
where $\chi^w$ is a characteristic function of the solvent region $\Omega_{\rm w}$.

To incorporate ionic steric effects in the framework of the Poisson--Boltzmann (PB) theory, we consider the electrostatic solvation free-energy functional of ionic concentrations:
\begin{equation}\label{Functional}
F(c) = F_{pot}(c)+ F_{ent}(c) + F_{ex}(c).
\end{equation}
Here the electrostatic potential solvation energy is given by
\[
F_{pot}(c) = \int_{\Omega_{\rm w}} \frac12 \rho \psi \, dV +\frac{1}{2}\sum_{i=1}^N Q_i\psi^r(x_i) -\frac{1}{2}\int_{\partial\Omega}\varepsilon_0\varepsilon_r\frac{\partial\psi}{\partial\bnu}\psi^{\infty}dS,
\]
where $\psi$ is determined by the BVP~\reff{BVP},  $\bnu$ denotes the unit exterior normal on $\partial \Omega$, and $\psi^r:= \psi-\psi^{f}$ is the reaction potential, with $\psi^{f}$ being the electrostatic potential of the same charged biomolecules before solvation:
\begin{equation}\label{psif}
	\psi^{f}(x)=\sum_{i=1}^N \frac{Q_i}{ 4\pi \ve_0\ve_{ m} |x- x_i|}.
\end{equation}
The first term in $F_{pot}(c)$ represents the electrostatic potential energy due to mobile ions, the second term describes the reaction field potential energy, and the last boundary term accounts for the electrostatic potential energy contributed from the exterior region. Further details can be found in the works~\cite{VISMPB_JCTC2014, LiuQiaoLu_SIAP18}.
The entropy of ions is given by
\[
F_{ent}(c) = \beta^{-1} \sum_{i=1}^M \int_{\Omega_{\rm w}} c_i \left[ \ln (v_i c_i) - 1 \right]\,dV,
\]
where $\beta = (k_B T)^{-1}$ with $k_B$ the Boltzmann constant and $T$ the temperature. The excess potential energy $F_{ex}(c)$ accounts for ionic steric effects. First variation of $F(c)$ with respect to $c_i$ leads to chemical potentials
\[
\mu_i:=\frac{\delta F}{\delta c_i} = z_i e \psi + \beta^{-1} \ln(v_i c_i) + \mu_i^{ex}(c), \quad i=1, \cdots, M.
\]
Here, the sum of first two terms is the chemical potential of ideal gas, and the excess chemical potential $\mu_i^{ex}:=\frac{\delta  F_{ex}}{\delta c_i}$ accounts for the ionic steric effects. See the references~\cite{Li_SIMA09, Li_Nonlinearity09, LiuQiaoLu_SIAP18} for formal calculations on the first variation.
In the equilibrium, the  total chemical potential can be further determined by the bulk ionic concentrations, $c^{\infty}=[c_1^\infty, \cdots, c_M^\infty]$, which are achieved when the electrostatic potential vanishes. Thus, we have equilibrium conditions for ionic concentrations:
\begin{equation}\label{EqCond}
z_i e \psi + \beta^{-1} \ln(\frac{c_i}{c_i^\infty}) + \mu_i^{ex}(c)- \mu_i^{ex}(c^{\infty})=0, \quad i=1, \cdots, M.
\end{equation}
Such nonlinear equilibrium conditions determine generalized Boltzmann distributions. 

Ions are treated as point charges in the mean-field approximations, which break down when the steric repulsion and correlation are no longer ignorable. To account for steric effects, one simple approach is using local density approximations, in which the excess chemical potential $\mu_i^{ex}$ is approximated as a local function of ionic concentrations~\cite{BazantReview_ACIS09}. Here, we briefly review several commonly used local models in literature. One popular model is based on the statistical mechanics of ions and solvent molecules on lattices~\cite{Bikerman:PM:1942,  BAO_PRL97, Li_Nonlinearity09, LuZhou_BiophysJ11}. The entropy of solvent molecules accounts for the excess chemical potential
\begin{equation}\label{LGmu}
\mu_i^{ex}= -\beta^{-1} \frac{v_i}{v_0} \ln \left(v_0 c_0 \right),
\end{equation}
where $v_0$ is the volume of a solvent molecule and $c_0$ is the solvent concentration defined by
\begin{equation}
\label{c0}
c_0(x) = v_0^{-1} \left[ 1 - \sum_{i=1}^M v_i c_i(x) \right].
\end{equation}
The hard-sphere interactions between ions can be described by the Lennard-Jones (LJ) potential energy, which gives rise to a nonlocal model~\cite{HyonLiuBob_CMS10}. To avoid computationally intractable integro-differential equations, local approximations of nonlocal integrals can be employed to obtain local models~\cite{HorngLinLiuBob_JPCB12, LinBob_CMS14, DingWangZhou_JCP19}. In such a local model, the excess chemical potential is given by 
\begin{equation}\label{LJmu}
\mu_i^{ex}=\beta^{-1} \sum_{j=1}^M \omega_{ij}c_j,
\end{equation}
where $(\omega_{ij})$ is a symmetric matrix to represent the cross interactions between different species and self interactions between ions of the same species.  Alternatively, the excess chemical potential can be determined by 
the Carnahan--Starling (CS) equation of state for hard-sphere liquids of a uniform size $v_i=v$~\cite{BazantReview_ACIS09}:
\begin{equation}\label{CSmu}
\mu_i^{ex}=\beta^{-1} \frac{\phi (8-9\phi+3\phi^2)}{(1-\phi)^3},
\end{equation}
where $\phi=v\sum_{i=1}^M c_i$ is volume fraction of ions. To extend the model to mixtures of unequal size, one can choose the Boublik--Mansoori--Carnahan--Starling--Leland (BMCSL) equation of state to derive the excess potential~\cite{BazantReview_ACIS09}.

For simplicity of presentation, we shall focus on the most popular local steric PB theory, which includes the BVP of the Poisson's equation~\reff{BVP} and the equilibrium conditions~\reff{EqCond} with the excess chemical potential given by~\reff{LGmu}, i.e.,
\begin{equation}\label{GovEqn}
z_{i} \beta e \psi+  \ln\left(\frac{c_{i}}{c_{i}^\infty}\right)-\frac{v_{i}}{v_{0}}\ln \left(\frac{1 - \sum_{j=1}^M v_j c_j }{1 - \sum_{j=1}^M v_j c_j^\infty }\right) =0, ~~1 \leq i \leq M, \qquad   \mbox{in }  \Omega_{\rm w}.
\end{equation}
We remark that the numerical methods presented below are applicable to other local steric PB theories with the excess chemical potential derived from the approximation of LJ potential~\reff{LJmu}, the CS equation of state~\reff{CSmu}, or the BMCSL equation of state. 


\section{Numerical Methods}\label{s:NumMethods}
\subsection{Discretization}\label{ss:Discretization}
We choose a computational domain $\Omega=[-L, L]^3$, and cover the domain with a uniform mesh
$$\Omega_h=\Big\{(x_i,y_j,z_k) | x_i=-L+ ih,~ y_j=-L + jh, ~z_k=-L+ kh,~ i, j, k=0,\dots,N_h+1 \Big\},$$
where $L$ is a positive number and $h=2L/(N_h+1)$ is the uniform grid spacing. Define 
\[
\mathring{\Omega}_h = \Big\{(x_i,y_j,z_k) \in \Omega_h | 1\leq i \leq N_h,  1\leq j \leq N_h, 1\leq k \leq N_h\Big\},
\]
and $ \partial\Omega_h= \Omega_h \backslash \mathring{\Omega}_h$. For a given known function $u(x, y, z)$, we denote $u_{i,j,k}=u(x_i, y_j, z_k)$. For an unknown function $v(x, y, z)$, we denote $v_{i,j,k}\approx v(x_i, y_j, z_k)$ as the numerical approximation of $v$ on the grid point $(x_i, y_j, z_k)$.  


The dielectric coefficient $\ve_r (x)$ determined by a given molecular surface $\Gamma$ has a sharp transition. The numerical instability in the calculations of solvation energies and forces, due to relative location and orientation of biomolecules with respect to the finite-difference grid, has been realized in literature~\cite{JWangRLuo_JCTC12, IBR:CPC:1998}. Mesh refinement could be a solution, but too small mesh resolution is still not affordable in 3D calculations. One approach to reduce the instability due to grid dependence is to smooth the sharp transition of dielectric coefficient across the solute and solvent regions~\cite{DavisMcCammon_JCC91, BruccoleriSharp_JCC97, JWangRLuo_JCTC12, APBS_Baker18, ChakravortyAlexov_JCTC2018, ChakravortyAlexov_FMB2018}. We here introduce a smoothed characteristic function of the solvent region:
\[
\chi^{w, \tau}(x) :=  H_{\tau}[\varphi(x)],
\]
where $\varphi(x)$ is a signed-distance level-set function whose zero level set represents the solute-solvent interface $\Gamma$, and $H_{\tau}(\cdot)$ is a smeared Heaviside function 
\[
H_{\tau}(s) =\left\{
\begin{aligned}
&1~ & &\text{if }~ s> \tau, \\
&\frac12 +\frac{s}{2\tau} +\frac{1}{2\pi} \sin \left(\frac{\pi s}{\tau}\right)~ & &\text{if }~ -\tau \leq s \leq \tau, \\
&0~ & &\text{if }~ s<- \tau.
\end{aligned}
\right.
\]
Here $\tau>0$ is a parameter to control the width of transition layer. Then the dielectric coefficient function becomes
\[
\ve^{r, \tau} (x) = (1- \chi^{w, \tau}) \ve_{\rm m} + \chi^{w, \tau} \ve_{\rm w}.
\]
With such a smooth transition, the BVP of the Poisson's equation for the electrostatic potential $\psi$ reads
\begin{equation}\label{Psi-smooth}
	\left\{
	\begin{aligned}
	& -\nabla  \cdot \ve_0 \ve^{r, \tau} \nabla \psi =\chi^{w, \tau}\sum_{j=1}^{M}z_{j}ec_{j} + \rho^f  \qquad & & \mbox{in } \Omega, \\
	& \psi = \psi^\infty  \qquad & & \mbox{on } \partial \Omega.
	\end{aligned}
	\right.
\end{equation}
Due to the fixed charges represented by Dirac delta functions, the electrostatic potential $\psi$ has singularities inside solute atoms. To get accurate approximations of $\psi$, we consider first solving the reaction potential $\psi^r = \psi - \psi^f$ instead. It follows from~\reff{psif} that
\[
- \ve_0 \ve_{\rm m} \Delta \psi^f=\rho^f.
\]
Therefore, $\psi^r$ satisfies the BVP
\begin{equation}\label{Psir}
	\left\{
	\begin{aligned}
	& -\nabla  \cdot \ve_0 \ve^{r, \tau} \nabla \psi^r =\chi^{w, \tau}\sum_{j=1}^{M}z_{j}ec_{j} + \nabla  \cdot \ve_0 (\ve^{r, \tau} -\ve_{\rm m}) \nabla \psi^f   \qquad & & \mbox{in } \Omega, \\
	& \psi^r = \psi^\infty - \psi^f   \qquad & & \mbox{on } \partial \Omega.
	\end{aligned}
	\right.
\end{equation} 
With the standard central differencing, the BVP~\reff{Psir} is approximated by
\begin{equation}\label{Scheme}
	\left\{
	\begin{aligned}
	& L_h \psi^r_{i,j,k} =\chi^{w, \tau}_{i,j,k} \sum_{l=1}^{M}z_{l}ec^{l}_{i,j,k} + \nabla_h  \cdot \ve_0 (\ve^{r, \tau}_{i,j,k} -\ve_{\rm m}) \nabla_h \psi^f_{i,j,k}   \qquad & & \mbox{in } \mathring{\Omega}_h, \\
	& \psi^r_{i,j,k} = \psi^\infty_{i,j,k} - \psi^f_{i,j,k}   \qquad & & \mbox{on } \partial \Omega_h,
	\end{aligned}
	\right.
\end{equation} 
where $c^{l}_{i,j,k}$ $(l=1,\cdots, M)$ is the numerical approximation of the concentration of the $l$th species, and the discrete Laplacian operator is defined by
\begin{equation}\label{Lh}
\begin{split}
L_h \psi^r_{i,j,k} = -\nabla_h  \cdot \ve_0 \ve^{r, \tau}_{i,j,k} \nabla_h \psi^r_{i,j,k}=&-\frac{\ve_0}{h^2}\Big[\ve^{r, \tau}_{i+\frac{1}{2},j,k}(\psi^r_{i+1,j,k}-\psi^r_{i,j,k})-
\ve^{r, \tau}_{i-\frac{1}{2},j,k}(\psi^r_{i,j,k}- \psi^r_{i-1,j,k}) \\&+\ve^{r, \tau}_{i,j+\frac{1}{2},k}(\psi^r_{i,j+1,k}-\psi^r_{i,j,k})-\ve^{r, \tau}_{i,j-\frac{1}{2},k}(\psi^r_{i,j,k} - \psi^r_{i,j-1,k})\\
&+\ve^{r, \tau}_{i,j,k+\frac{1}{2}}(\psi^r_{i,j,k+1}-\psi^r_{i,j,k})- \ve^{r, \tau}_{i,j,k-\frac{1}{2}}(\psi^r_{i,j,k}- \psi^r_{i,j,k-1})\Big].
\end{split}
\end{equation} 
The term $\nabla_h  \cdot \ve_0 (\ve^{r, \tau}_{i,j,k} -\ve_{\rm m}) \nabla_h \psi^f_{i,j,k}$ is defined analogously. To improve the convergence of reaction field energies~\cite{DavisMcCammon_JCC91, BruccoleriSharp_JCC97}, dielectric coefficients on half grid points are approximated by the harmonic average
\[
\ve^{r, \tau}_{i+\frac{1}{2},j,k}= \frac{2\ve^{r, \tau}_{i,j,k} \ve^{r, \tau}_{i+1,j,k}}{\ve^{r, \tau}_{i,j,k} + \ve^{r, \tau}_{i+1,j,k}}.
\]
Given concentrations $c^{l}_{i,j,k}$, the electrostatic potential can be obtained by $\psi_{i,j,k}=\psi^r_{i,j,k} + \psi^f_{i,j,k}$, where the reaction potential $\psi^r_{i,j,k}$ is obtained by solving the linear system~\reff{Scheme}. 
The nonlinear equilibrium conditions~\reff{GovEqn} can be discretized as 
\begin{equation}\label{R=0}
z_{l} \beta e (\psi^r_{i,j,k} + \psi^f_{i,j,k}) +  \ln\left(\frac{c^{l}_{i,j,k}}{c^{l,\infty}_{i,j,k}}\right)-\frac{v_{l}}{v_{0}}\ln \left(\frac{1 - \sum_{m=1}^M v_{m} c^m_{i,j,k} }{1 - \sum_{m=1}^M v_{m} c^{m, \infty}_{i,j,k}} \right) =0,~~ l=1, \cdots, M, ~~\mbox{in } \Omega_h.
\end{equation}
In summary, we have the following nonlinear system after discretization:
\begin{equation}\label{FullSys}
	\left\{
	\begin{aligned}
	&z_{l} \beta e (\psi^r_{i,j,k} + \psi^f_{i,j,k}) +  \ln\left(\frac{c^{l}_{i,j,k}}{c^{l,\infty}_{i,j,k}}\right)-\frac{v_{l}}{v_{0}}\ln \left(\frac{1 - \sum_{m=1}^M v_{m} c^m_{i,j,k} }{1 - \sum_{m=1}^M v_{m} c^{m, \infty}_{i,j,k}} \right) =0,~l=1,..., M, & & \mbox{in } \Omega_h, \\
	& L_h \psi^r_{i,j,k} =\chi^{w, \tau}_{i,j,k} \sum_{l=1}^{M}z_{l}ec^{l}_{i,j,k} + \nabla_h  \cdot \ve_0 (\ve^{r, \tau}_{i,j,k} -\ve_{\rm m}) \nabla_h \psi^f_{i,j,k}  & & \mbox{in } \mathring{\Omega}_h, \\
	& \psi^r_{i,j,k} = \psi^\infty_{i,j,k} - \psi^f_{i,j,k}  & & \mbox{on } \partial \Omega_h.
	\end{aligned}
	\right.
\end{equation} 

Numerical iterative methods for solving the coupled system~\reff{FullSys} have been proposed in literature to obtain the unknowns $\{c^{1}_{i,j,k}, \cdots, c^{M}_{i,j,k}, \psi^r_{i,j,k}\}_{i,j,k=1,\cdots, N_h}$. For instance, a nonlinear successive over-relaxation (SOR) scheme has been developed in~\cite{Xie_IJNAM2017}, in which the linear system~\reff{Scheme} was solved with concentrations given in a previous iteration step, and the nonlinear system~\reff{R=0} was updated with a Newton iteration method for each species in a Gauss--Seidel manner, with the electrostatic potential and un-updated concentrations of other species given in a previous step. The unknowns on grid points were coupled together in the SOR iterations. 
As observed in the work~\cite{JiZhou_CMS19}, the unknowns of concentrations in~\reff{R=0} can be \emph{spatially decoupled} for each grid point with a given electrostatic potential. Based on such an observation, efficient and memory-saving iteration methods for~\reff{R=0} on every single grid point $(x_i,y_j,z_k)$ can be designed. 
In this work, we shall propose a novel and super efficient Newton iteration method with truncation for the coupled system~\reff{FullSys}, by treating the concentrations as implicit scalar functions of the electrostatic potential, i.e., generalized Boltzmann distributions.  
\subsection{Generalized Boltzmann Distribution}
This section focuses on solving the nonlinear equilibrium conditions~\reff{GovEqn} with a given electrostatic potential $\psi$, to obtain generalized Boltzmann distributions. The existence, uniqueness, boundness, and smoothness of such distributions are established in the following theorem. 
\begin{theorem}
\label{t:GBD}
For each $\psi \in \R$,  there exists a unique solution $c_i = B_i (\psi)$ $(i = 1, \dots, M)$ to the equilibrium conditions~\reff{GovEqn}.  Moreover, each generalized Boltzmann distribution $B_i(\psi): \R \to (0, v_i^{-1})$ is a smooth function.
\end{theorem}
\begin{proof}

The nonlinear equilibrium conditions~\reff{GovEqn} can be rewritten as
\begin{equation}\label{ImpBD}
c_i= c_i^{\infty} \left(\frac{\gamma}{\gamma^\infty} \right)^{v_i/v_0} \exp(-\beta z_i e \psi), ~~i = 1, \cdots, M,
\end{equation}
with the volume fraction of solvent $\gamma$ defined by~\cite{LiLiuXuZhou_Nonlinearity13}
\begin{equation}\label{gamma}
\gamma=1-\sum_{j=1}^M v_j c_j.
\end{equation}
Analogously, $\gamma^{\infty} = 1-\sum_{j=1}^M v_j c_j^\infty$, where bulk concentrations $c_j^\infty$ are chosen so that $\gamma^{\infty} \in (0, 1)$. A combination of~\reff{ImpBD} and~\reff{gamma} leads to an equation for $\gamma$:
\begin{equation}\label{fgamma}
f(\gamma):= \gamma -1 + \sum_{j=1}^M v_j c_j^\infty \left(\frac{\gamma}{\gamma^\infty} \right)^{v_j/v_0} \exp(-\beta z_j e \psi) = 0.
\end{equation}
There is a unique root for $f(\cdot)$ in the interval $(0, 1)$. Indeed, it is easy to verify that $f(0) =-1 <0$ and 
\[
f(1)=\sum_{j=1}^M v_j c_j^\infty \left(\frac{1}{\gamma^\infty} \right)^{v_j/v_0} \exp(-\beta z_j e \psi) >0.
\]
Simple calculations show that
\[
f'(\gamma) = 1+  \sum_{j=1}^M \frac{v_j^2 c_j^\infty \gamma^{v_j/v_0-1}}{v_0  \left({\gamma^\infty} \right)^{v_j/v_0} }   \exp(-\beta z_j e \psi)  >0,~~ \forall \gamma \in (0, 1).
\]
Therefore, $f(\cdot)$ is a continuous, increasing function with a unique root in $(0, 1)$. It follows that for each $\psi \in \R$, there exists a unique corresponding $\gamma \in (0, 1)$. This establishes the function $\gamma=\gamma(\psi): \R \to (0, 1)$. Further by~\reff{ImpBD}, one can see that there exists a unique $c_i$ for each $\psi \in \R$. For each $\psi \in \R$, we define functions $B_i(\psi) = c_i$, $i =1, \cdots, M$, i.e., the generalized Boltzmann distributions.  It follows from~\reff{ImpBD} and \reff{gamma} with $\gamma \in (0, 1)$ that $c_i=B_i(\psi) \in (0, v_i^{-1})$. From the equation~\reff{fgamma}, one can show by the implicit function theorem that $\gamma(\psi)$ is a smooth function. Therefore, by~\reff{ImpBD}, $B_i(\psi)$ is a smooth function as well.    
 \end{proof}
For each $\psi \in \R$, the corresponding $\gamma$ can be calculated efficiently with a Newton iteration scheme
\begin{equation}\label{Nt4f}
\gamma^{l+1}=\gamma^{l} - \frac{f(\gamma^l)}{f'(\gamma^l)}.
\end{equation}
With the obtained $\gamma$, the numerical concentrations for each $\psi$ can be calculated via~\reff{ImpBD}.

\subsection{Precomputing Speed-Up Strategy}\label{ss:Precomput}
With the implicit generalized Boltzmann distributions $B_1(\psi), \cdots, B_M(\psi)$ that are numerically available for each given $\psi$, the unknowns $\{c^{1}_{i,j,k}, \cdots, c^{M}_{i,j,k}, \psi^r_{i,j,k}\}$ on $\Omega_h$ governed by the coupled system~\reff{FullSys} can be obtained more efficiently by solving
\begin{equation}\label{GBDScheme}
	\left\{
	\begin{aligned}
	& L_h \psi^r_{i,j,k} =\chi^{w, \tau}_{i,j,k} \sum_{l=1}^{M}z_{l}e B_{l}(\psi^f_{i,j,k}+\psi^r_{i,j,k}) + \nabla_h  \cdot \ve_0 (\ve^{r, \tau}_{i,j,k} -\ve_{\rm m}) \nabla_h \psi^f_{i,j,k} ~ & & \mbox{in } \mathring{\Omega}_h, \\
	& \psi^r_{i,j,k} = \psi^\infty_{i,j,k} - \psi^f_{i,j,k} ~ & & \mbox{on } \partial \Omega_h.
	\end{aligned}
	\right.
\end{equation}
Here, instead of $\{c^{1}_{i,j,k}, \cdots, c^{M}_{i,j,k}, \psi^r_{i,j,k}\}$, $\{\psi^r_{i,j,k}\}$ is served as the iterative variable.
To solve the nonlinear system~\reff{GBDScheme} iteratively, we need to evaluate for many times the generalized Boltzmann distributions $B_l(\cdot)$ and  $B'_l (\cdot)$, if Newton-type iteration methods are considered.   

To speed-up computations, we propose to precompute the functions $B_l(\psi)$ and $B'_l (\psi)$ for $l=1, \cdots, M$, and store the functions on a mesh of $\psi$ in a certain interval. For instance, we consider $\psi \in [\psi_L, \psi_R]$ and assume that the interval is large enough to cover the range we are interested in. The two end points $\psi_L$ and $\psi_R$ can be estimated in specific applications; See~\reff{Spsi} for their estimation. We cover the interval with a mesh $p_i=\psi_L+i h_\psi$ for $i=0, \cdots, N_\psi$, where the mesh spacing $h_\psi=(\psi_R-\psi_L)/N_\psi$.  For each grid point $\psi=p_i$, we first compute $\gamma(p_i)$ with the iteration scheme~\reff{Nt4f}, then compute $B_l(p_i)=c_l$ with~\reff{ImpBD}. Note that $\gamma(p_i)$ can be an  initial guess for the iterations~\reff{Nt4f} when computing $\gamma(p_{i+1})$. Such a continuation method can provide very good initial guesses for the Newton iterations. 

For each $p_i$, it follows from~\reff{ImpBD} that 
\begin{equation}\label{CPrime}
B'_l (p_i) = \left( \frac{v_l \gamma'(p_i)}{v_0 \gamma(p_i)} - \beta z_l e \right) B_l(p_i), 
\end{equation}
where
\begin{equation}\label{GammaPrime}
\gamma' (p_i) =  \frac{\beta v_0 e \gamma(p_i) \sum_{l=1}^M z_l v_l B_l (p_i) }{v_0 \gamma(p_i) + \sum_{l=1}^M v_l^2 B_l (p_i) }.
\end{equation}
By Theorem~\ref{t:GBD}, we know that $B_l(\cdot)$ and $\gamma(\cdot)$ are smooth functions. Thus, we can  compute the derivatives alternatively with high-order difference schemes, e.g.,
\[
B'_l (p_i) \approx \frac{-B_l(p_{i+2}) +8 B_l(p_{i+1}) - 8B_l(p_{i-1}) +B_l(p_{i-2})}{12 h_\psi},
\]
with the stored data $\left\{B_l(p_i)\right\}_{i=0}^{N_\psi}$.  Therefore, we can precompute and store the vectors $\left\{B_l(p_i)\right\}_{i=0}^{N_\psi}$ and $\left\{B_l'(p_i)\right\}_{i=0}^{N_\psi}$ for $l=1, \cdots, M$.   To save memory, we can alternatively store the vectors $\left\{\gamma(p_i)\right\}_{i=0}^{N_\psi}$ and $\left\{\gamma'(p_i)\right\}_{i=0}^{N_\psi}$ instead, especially when the number of ionic species $M$ under consideration is large. With $\left\{\gamma(p_i)\right\}_{i=0}^{N_\psi}$ and $\left\{\gamma'(p_i)\right\}_{i=0}^{N_\psi}$ stored, we then can compute $\left\{B_l(p_i)\right\}_{i=0}^{N_\psi}$ and $\left\{B_l'(p_i)\right\}_{i=0}^{N_\psi}$ directly with~\reff{ImpBD} and~\reff{CPrime}, respectively.

By the smoothness of  $B_l(\cdot)$ and $\gamma(\cdot)$, we can use high-order interpolation schemes to interpolate  $B_{l}(\psi)$ and $B'_{l}(\psi)$ for any given $\psi$ with stored data $\left\{B_l(p_i)\right\}_{i=0}^{N_\psi}$ and $\left\{B_l'(p_i)\right\}_{i=0}^{N_\psi}$, when solving~\reff{GBDScheme} iteratively.  We remark that such a precomputing-interpolation strategy is applicable to other local steric PB theories with implicit generalized Boltzmann distributions $B_1(\psi), \cdots, B_M(\psi)$ determined by equilibrium conditions~\reff{EqCond}, in which the excess chemical potential can be derived from the approximation of LJ potential~\reff{LJmu}, the CS equation of state~\reff{CSmu}, or the BMCSL equation of state.  Furthermore, such a precomputing and interpolation strategy makes our algorithm for solving steric PB theories as efficient as for solving the classical PB theory. 
\subsection{Newton Iteration Method with Truncation}
The nonlinear system~\reff{GBDScheme} can be rewritten in a matrix form
\begin{equation}\label{matrix form}
A {\bpsi} =G({\bpsi}),
\end{equation}
where $\bpsi=(\psi^r_{1,1,1}, \psi^r_{2,1,1}, \cdots, \psi^r_{N_h,N_h,N_h})^T$ is the unknown vector, $A$ is the coefficient matrix corresponding to the discrete Laplacian, and $G({\bpsi})=(g(\psi^r_{1,1,1}), g(\psi^r_{2,1,1}), \cdots, g(\psi^r_{N_h,N_h,N_h}))^T+\bfb$ is a column vector. Here $g(\psi^r_{i,j,k}):=\chi^{w, \tau}_{i,j,k} \sum_{l=1}^{M}z_{l}e B_{l}(\psi^f_{i,j,k}+\psi^r_{i,j,k})$ and the column vector $\bfb$ results from boundary conditions and the known term $ \nabla_h  \cdot \ve_0 (\ve^{r, \tau}_{i,j,k} -\ve_{\rm m}) \nabla_h \psi^f_{i,j,k}$. It is standard to show that $A$ is a symmetric positive definite matrix with positive diagonal elements and negative off-diagonal elements. Define the residual function $F(\cdot): \R^{N_h^3} \to \R^{N_h^3}$ by $ F(\bpsi): =A\bpsi -G(\bpsi)$, and its Jacobian matrix $F'(\cdot): \R^{N_h^3} \to \R^{N_h^3 \times N_h^3}$ by $F'(\bpsi)=A-G'(\bpsi)$, where $G'(\bpsi)=\text{{\rm diag}} (g'(\psi^r_{1,1,1}), \cdots, g'(\psi^r_{N_h,N_h,N_h}))$. The existence and uniqueness of the solution to the nonlinear system~\reff{matrix form} can be established in the following theorem.
\begin{theorem}\label{Existence}
There exists a unique solution to the nonlinear system~\reff{matrix form}.
\end{theorem}
\begin{proof}
Define an energy $E: \R^{N_h^3} \to \R$ by
\begin{equation}\label{E}
E(\bita) = \frac12 \bita^{T} A  \bita - \sum_{l=1}^{N_h^3} \int_{0}^{\eta_l} g(s) ds- \bfb^T\bita,~ \forall  \bita=(\eta_1, \eta_2, \cdots, \eta_{N_h^3})^T \in \R^{N_h^3}.
\end{equation}
By direction calculations, we have 
\begin{equation*}
\nabla_{\bita}E (\bita) =F(\bita) =  A \bita - G(\bita)~\text{and} ~ \nabla_{\bita}^2 E (\bita) = F'(\bita) = A - G'(\bita).
\end{equation*}
We next prove that the function $E(\eta)$ is a convex function. It follows from~\reff{CPrime} and \reff{GammaPrime} that for any $\psi \in \R$,
\begin{equation*}
	\begin{split}
	\sum\limits_{l=1}^M z_lB_l'(\psi)&=\frac{1}{v_0 \gamma(\psi)}\sum\limits_{l=1}^M\left[z_lv_lB_l\gamma'(\psi)-\beta e v_0\gamma(\psi) z_l^2 B_l\right]\\
	&=-\frac{\beta e}{v_0 \gamma(\psi)+\sum\limits_{j=1}^M v^2_jB_j}\left[v_0\gamma(\psi)\sum\limits_{j=1}^M z_j^2 B_j+\left(\sum\limits_{l=1}^M z_l^2 B_l\right) \left(\sum\limits_{j=1}^M v_j^2 B_j\right)-\left(\sum\limits_{l=1}^M z_lv_lB_l\right)^2\right].
	\end{split}
\end{equation*}
By the Cauchy--Schwarz inequality,  we have
\begin{equation}\label{csum>=0}
	\sum\limits_{l=1}^M z_lB_l'(\psi)\leq-\frac{\beta e v_0\gamma(\psi)}{v_0 \gamma(\psi)+\sum\limits_{j=1}^M v^2_jB_j} \sum\limits_{j=1}^M z_j^2 B_j\leq 0.
\end{equation}
Therefore, for any $\bd=(d_1, \cdots, d_{N^3_h})^T \in \R^{N^3_h}$,
\[
\bd^T \nabla_{\bita}^2 E (\bita) \bd = \bd^T A \bd - \sum_{m=1}^{N^3_h} g'(\eta_{m}) d_{m}^2 \geq 0,
\]
where the positive definiteness of $A$, the definition of $g'$, and the conclusion of~\reff{csum>=0} have been used in the last step. Hence, the function $E(\bita)$ is a convex function. Therefore, there exists a unique minimizer $\bpsi$ of $E(\bita)$ satisfying $\nabla_{\bita}E (\bpsi) = 0$, i.e., the nonlinear system~\reff{matrix form}. 
\end{proof}
\subsubsection{Upper and Lower Estimates}
It follows from Theorem~\ref{t:GBD} that $0<c_l=B_l(\psi)<v_l^{-1}$ and $\gamma (\psi) \in (0, 1)$ for $\psi \in \R$. Such bounds can help establish upper and lower estimates for the unknown $\bpsi$. Define grid functions $u^+$ and $u^-$ by the following problems
\begin{equation}\label{DefU}
  \left\{\begin{aligned}
  &L_h u_{i,j,k}^+=  \chi^{w, \tau}_{i,j,k} e \max \limits_{1\leq l \leq M} \frac{ z_l }{v_l}  + b_{i,j,k} & \mbox{in }  \mathring{\Omega}_h,\\
  &u_{i,j,k}^{+}=  \psi^\infty_{i,j,k} - \psi^f_{i,j,k}, &  \mbox{on }  \partial\Omega_h,
  \end{aligned} \right.
  \quad
  \left\{\begin{aligned}
  &L_h u_{i,j,k}^-= \chi^{w, \tau}_{i,j,k} e \min \limits_{1\leq l \leq M} \frac{ z_l }{v_l}   +b_{i,j,k} & \mbox{in } \mathring{\Omega}_h,\\
  &u_{i,j,k}^-= \psi^\infty_{i,j,k} - \psi^f_{i,j,k} & \mbox{on } \partial\Omega_h,
  \end{aligned}\right.
\end{equation}
where $b_{i,j,k}$ are components of the vector $\bfb$. From the discretization of the Laplacian, we have the following standard discrete extremum principle without giving its proof.
\begin{lemma}\label{t:ExtPrin}(Extremum Principle)
 Let $u$ be a grid function defined on $\Omega_h$. Let $L_h$ be the discrete Laplacian operator defined by~\reff{Lh}. Then
\begin{enumerate}[(a)]
  \item If $u$ satisfies $L_h u_{i,j,k}\leq 0$, $(x_i,y_j,z_k)\in \mathring{\Omega}_h$, then $\max\limits_{\Omega_h} u_{i,j,k}\leq\max\limits_{\partial\Omega_h} u_{i,j,k}$.
  \item If $u$ satisfies $L_h u_{i,j,k}\geq 0$, $(x_i,y_j,z_k)\in \mathring{\Omega}_h$, then $\min\limits_{\Omega_h} u_{i,j,k}\geq\min\limits_{\partial\Omega_h} u_{i,j,k}$.
\end{enumerate}
\end{lemma}
Clearly, we have 
\[
g(\psi^r_{i,j,k})
                      = \chi^{w, \tau}_{i,j,k} \sum_{l=1}^{M}\frac{z_{l}}{v_l} e v_l B_{l}(\psi^f_{i,j,k}+\psi^r_{i,j,k}) \in (\chi^{w, \tau}_{i,j,k} e \min \limits_{1\leq l \leq M} \frac{ z_l }{v_l} ,~ \chi^{w, \tau}_{i,j,k} e \max \limits_{1\leq l \leq M} \frac{ z_l }{v_l}),
\]
where the fact that the volume fraction $v_l B_l \in (0, 1)$ has been used. By the extremum principle, we have the solution $\bpsi \in S_{\psi}$, where  the set $S_{\psi}$ is given by
\begin{equation}\label{Spsi}
  S_{\psi}=\left\{\bpsi\in \R^{N_h^3} | \bu^- \preceq  \bpsi \preceq \bu^+ \right\},
\end{equation}
with $\bu^{\pm}=(u^\pm_{1,1,1}, u^\pm_{2,1,1}, \cdots, u^\pm_{N_h,N_h,N_h})^T$.
We now propose the following Newton iteration method with truncation for $F(\bpsi) =0$, with computed upper and lower estimates $\bu^+$ and $\bu^-$.
\begin{algorithm} 
\caption{Newton Iterative method with Truncation}
\label{alg:tnewtonsor}
\begin{algorithmic}[1]
\STATE  
Compute the upper and lower estimates $\bu^+$ and $\bu^-$. Initialize $\bpsi^{(0)}$ so that $\bu^- \preceq  \bpsi^{(0)} \preceq \bu^+$. Let the iteration step $n=0$. Choose a stopping tolerance $tol$. 
\WHILE{$\|F(\bpsi^{(n)}) \| > tol$}{ 
\STATE Compute $\bdelta^{(n)}=-\omega F'(\bpsi^{(n)})^{-1} F(\bpsi^{(n)})$, where $\omega \in (0, 1]$ is determined by backtracking so that $$\|F(\bpsi^{(n)}+\bdelta^{(n)}) \| \leq \|F(\bpsi^{(n)}) \|$$

\STATE Update $\overline{\bpsi}^{(n+1)}=\bpsi^{(n)}+\bdelta^{(n)}$

\STATE    
Find $\bpsi^{(n+1)}$ by truncations
  \begin{equation}\label{TruncationStep}
  \bpsi_{m}^{(n+1)}=\left\{
  \begin{aligned}
  \bu^+_{m} &\qquad\text{if }~\overline{\bpsi}_{m}^{(n+1)}>\bu^+_{m}, \\
  \overline{\bpsi}_{m}^{(n+1)} &\qquad\text{if }~ \bu^-_{m} \leq \overline{\bpsi}_{m}^{(n+1)}\leq \bu^+_{m},\\
  \bu^-_{m} &\qquad\text{if }~ \overline{\bpsi}_{m}^{(n+1)}<\bu^-_{m}, 
  \end{aligned}
  \right.
  \end{equation}
  for $m=1,2,...,N_h^3$ being the component index of corresponding vectors. 
 \STATE    Update $n \leftarrow n+1$ and $\bpsi^{(n)} \leftarrow \bpsi^{(n+1)}$.
}
\ENDWHILE 
\end{algorithmic}
\end{algorithm}	

\begin{remark}
Note that the proposed algorithm for the coupled system~\reff{Scheme} and \reff{R=0} is much more efficient than algorithms using $\{c^{1}_{i,j,k}, \cdots, c^{M}_{i,j,k}, \psi^r_{i,j,k}\}_{i,j,k=1,\cdots, N_h}$ as the iterative unknowns. The advantage is achieved due to several aspects: First, the unknown variable of the proposed algorithm only involves the electrostatic potential, reducing memory from $\mathcal{O} \left((M+1)N^3\right)$ to $\mathcal{O} (N^3)$; Second, the Newton iteration method with truncation uses the information on the derivative of concentrations with respect to the electrostatic potential, which makes the super-linear convergence of the whole algorithm possible; Third, the precomputing and interpolation strategy detailed in Section~\ref{ss:Precomput} can further speed up the calculation of the matrix of $F'(\bpsi^{(n)})$. The precomputing and interpolation strategy can be applied to solve other local steric PB theories mentioned in Section~\ref{s:Theory}, and makes the proposed algorithm for solving steric PB theories as efficient as for solving the classical PB theory. 
\end{remark}

\subsubsection{Algorithm Analysis}
This subsection presents numerical analysis on the proposed Newton iteration method with truncation.  In each Newton iteration step, a linear system involving the Jacobian matrix $F'(\bpsi^{(n)})$ needs to be solved. Its solvability is guaranteed by the following Lemma.
\begin{lemma}\label{FInvertibility}
The Jacobian matrix $F'(\bpsi)=A-G'(\bpsi)$ in Algorithm~\ref{alg:tnewtonsor} is invertible for any $\bpsi$.
\end{lemma}  
\begin{proof}
It follows from~\reff{csum>=0} that 
\[
g'(\psi^r_{i,j,k})=\chi^{w, \tau}_{i,j,k} \sum_{l=1}^{M}z_{l}e B'_{l}(\psi^f_{i,j,k}+\psi^r_{i,j,k}) \leq 0.
\]
By the facts that $G'(\bpsi)=\text{{\rm diag}} (g'(\psi^r_{1,1,1}), \cdots, g'(\psi^r_{N_h,N_h,N_h}))$ and $A$ is a symmetric positive definite matrix corresponding to $L_h$, we obtain that the Jacobian matrix $F'(\bpsi)=A-G'(\bpsi)$ is a symmetric positive definite matrix as well. This completes the proof.
\end{proof}
For the truncation step in the algorithm, we have the following  Lemma.
\begin{lemma}{\label{Truncation}}
The truncation step in~\reff{TruncationStep} has the following properties:
\begin{compactenum}
\item[\mbox{\rm (1)}]
The energy is decreasing: $$E(\overline{\bpsi}^{(n+1)}) \geq  E(\bpsi^{(n+1)});$$
\item[\mbox{\rm (2)}]
The error is decreasing: $$\left| \overline{\psi}_{m}^{(n+1)}-\psi_{m}^* \right| \geq \left| \psi_{m}^{(n+1)}-\psi_{m}^* \right|,$$
where $m=1,2,...,N_h^3$ and  $\bpsi^{*}$ is the exact solution, i.e., $F(\bpsi^{*}) = 0$.
\end{compactenum} 
\end{lemma} 
\begin{proof}
(1) By the Taylor's theorem, we have 
\[
\begin{aligned}
E(\overline{\bpsi}^{(n+1)})  - E(\bpsi^{(n+1)}) = &(\overline{\bpsi}^{(n+1)}- \bpsi^{(n+1)})^T \nabla E( \bpsi^{(n+1)}) \\
&+ \frac12  (\overline{\bpsi}^{(n+1)}- \bpsi^{(n+1)})^T \nabla^2 E (\bxi_1) (\overline{\bpsi}^{(n+1)}- \bpsi^{(n+1)}),
\end{aligned}
\]
where $\bxi_1$ is a vector between $\overline{\bpsi}^{(n+1)}$ and $\bpsi^{(n+1)}$ in a component-wise sense. By the convexity of $E$, we have
\[
E(\overline{\bpsi}^{(n+1)})  - E(\bpsi^{(n+1)}) \geq  (\overline{\bpsi}^{(n+1)}- \bpsi^{(n+1)})^T \nabla E( \bpsi^{(n+1)}). 
\]
We first consider the upper truncation in~\reff{TruncationStep} by $\bu^+$.  Denote $$\mathcal{P} = \left \{m | \overline{\bpsi}^{(n+1)}_{m}  > \bu^+_{m},  m=1, \cdots, N_h^3 \right\},$$ and 
$$\mathcal{N}_{m} = \left \{ j | j= m \pm 1, m\pm N_h, m\pm N_h^2,~ \text{and }  1\leq j \leq N_h^3 \right\} $$ as the set of indices of grid points in $\mathring{\Omega}_h$ neighboring to the grid point with the index $m$. By the upper truncation step in~\reff{TruncationStep}, we obtain 
\begin{equation}\label{Ediff}
(\overline{\bpsi}^{(n+1)}- \bpsi^{(n+1)})^T \nabla E( \bpsi^{(n+1)}) =\sum_{m\in \mathcal{P}} (\overline{\bpsi}^{(n+1)}_{m}  - \bu^+_{m})  r_{ m},
\end{equation}
where the residual
\[
r_{m}=  \sum_{j \in \mathcal{N}_{m} } a_{mj} \bpsi^{(n+1)}_j  + a_{mm} \bpsi^{(n+1)}_{m} - g(\bpsi^{(n+1)}_{ m}) - \bfb_{m},
\]
and $a_{mj}=(A)_{mj}$. 

 We next show that the residual $r_{m}$ is negative. It follows from the definition of $\bu^+$ that 
\begin{equation}\label{Def:u+}
  \sum_{j \in \mathcal{N}_{m} } a_{mj} \bu^{+}_j  + a_{mm} \bu^{+}_{m} -  \bchi_{m} e \max \limits_{1\leq l \leq M} \frac{ z_l }{v_l}  - \bfb_{m}  =0,
\end{equation}
where $\bchi_{m}$ represents the $m$-th component of the vector $\bchi=(\chi^{w, \tau}_{1,1,1}, \chi^{w, \tau}_{2,1,1}, \cdots, \chi^{w, \tau}_{N_h,N_h,N_h})^T$. Subtracting~\reff{Def:u+} from the residual $r_{m}$, we obtain
\[
r_{m}= \sum_{j \in \mathcal{N}_{m} } a_{mj} (\bpsi^{(n+1)}_j - \bu^{+}_j) - \left( g(\bpsi^{(n+1)}_{m}) -  \bchi_{m} e \max \limits_{1\leq l \leq M} \frac{ z_l }{v_l} \right),
\]
where $ \bpsi^{(n+1)}_{m}=\bu^{+}_{m}$, due to the truncation, has been used. By the facts that $a_{mj} <0$ for $m\neq j$, $\bpsi^{(n+1)}_j \leq \bu^{+}_j$, and $g(\bpsi^{(n+1)}_{m}) \leq \bchi_{m} e \max \limits_{1\leq l \leq M} \frac{ z_l }{v_l}$ (by Theorem~\ref{t:GBD}),  we have $r_{m}\geq 0.$ Thus, we have by~\reff{Ediff} that $E(\overline{\bpsi}^{(n+1)}) \geq  E(\bpsi^{(n+1)}).$ With an analogous analysis, we can show that the lower truncation by $\bu^-$ can further decrease the energy. This completes the proof.

(2) Let $\bpsi^{*}$ be the exact solution to the nonlinear system $F(\cdot) = 0$. As $\bpsi_{m}^* \in [\bu^-_{m}, \bu^+_{m}]$, we have 
\[
\overline{\bpsi}_{m}^{(n+1)}-\bpsi_{m}^*\geq \bpsi_{m}^{(n+1)}-\bpsi_{m}^*\geq0,~ \mbox{if } ~ \overline{\bpsi}_{m}^{(n+1)} > \bu^+_{m},
\]
and 
\[
\overline{\bpsi}_{m}^{(n+1)}-\bpsi_{m}^*\leq \bpsi_{m}^{(n+1)}-\bpsi_{m}^*\leq0,~ \mbox{if } ~ \overline{\bpsi}_{m}^{(n+1)} < \bu^-_{m},
\]
for $m=1,2,...,N_h^3.$ Thus, $\left| \overline{\psi}_{m}^{(n+1)}-\psi_{m}^* \right| \geq \left| \psi_{m}^{(n+1)}-\psi_{m}^* \right|.$
This completes the proof.
\end{proof}

After truncation, each iteration step is bounded by upper and lower estimates $\bu^+$ and $\bu^-$. The boundness leads to the following Lipschitz condition, which plays a key role in the algorithm analysis.
\begin{lemma}{\label{Lip}}
The matrix function $G'(\cdot)$ satisfies the Lipschitz condition on the bounded set $S_{\psi}$, i.e. there exists a positive constant $L$ such that
\begin{equation}\label{LIP}
\| G'(\bu)-G'(\bv)\| \leq L \| \bu-\bv\|,~ \forall \bu, \bv\in S_\psi.
\end{equation}
\end{lemma} 

\begin{proof}
It follows from the Theorem~\ref{t:GBD} that $B_l(\psi)$ is a smooth function. Thus, the matrix function $G'(\cdot)$ is also a smooth function by its definition.  Therefore, $G'(\cdot)$ satisfies the Lipschitz condition in the bounded set $S_{\psi}$ with the Lipschitz constant $L$ dependent on $\bu^{\pm}$.
\end{proof}
The stepsize $\omega$ is chosen so that the residual function decreases in the line search. The following Lemma ensures the existence of such a stepsize $\omega$.
\begin{lemma}{\label{Omega}}
There exists a stepsize $\omega \in \left(0, \dfrac{2\|F(\bpsi^{(n)})\|}{L\|F'(\bpsi^{(n)})^{-1}F(\bpsi^{(n)})\|^2} \right],$
such that $\|F(\overline{\bpsi}^{(n+1)})\| \leq \|F(\bpsi^{(n)})\| $, where $L$ is the Lipschitz constant in~\reff{LIP}. 
\end{lemma}
\begin{proof} 
By the Lipschitz condition~\reff{LIP},  we have
\begin{equation}\label{ResEst}
 \begin{split}
\|F(\overline{\bpsi}^{(n+1)})\|&=\|F(\overline{\bpsi}^{(n+1)})-F(\bpsi^{(n)})- F'(\bpsi^{(n)})(\overline{\bpsi}^{(n+1)}-\bpsi^{(n)})+(1-\omega) F(\bpsi^{(n)})\|  \\
&=\|\int_0^1 \left[F'\left(\bpsi^{(n)}+t(\overline{\bpsi}^{(n+1)}-\bpsi^{(n)})\right)-F'(\bpsi^{(n)})\right] (\overline{\bpsi}^{(n+1)}-\bpsi^{(n)})\,\mathrm{d}t +(1-\omega) F(\bpsi^{(n)})\| \\
&\leq \dfrac{L}{2}\|\overline{\bpsi}^{(n+1)}-\bpsi^{(n)}\|^2 +(1-\omega)\| F(\bpsi^{(n)}) \| \\
&= \dfrac{L\omega^2}{2}\|F'(\bpsi^{(n)})^{-1}F(\bpsi^{(n)})\|^2 +(1-\omega) \|F(\bpsi^{(n)}) \|. 
\end{split}
\end{equation}
Thus, if $0 < \omega \leq  \dfrac{2\|F(\bpsi^{(n)})\|}{L\|F'(\bpsi^{(n)})^{-1}F(\bpsi^{(n)})\|^2}$, we have $\|F(\overline{\bpsi}^{(n+1)})\| \leq \|F(\bpsi^{(n)})\| $.
\end{proof}
The local convergence for the proposed Newton iteration method with truncation can be established in the following theorem.

\begin{theorem}\label{thm:local}
Let $\{\bpsi^{(n)} \}$ be the sequence generated by the Algorithm~\ref{alg:tnewtonsor}. Let $\bpsi^*$ be the unique exact solution to the nonlinear equations $F(\cdot)=0$. If there exists some integer $k>0$ such that $\|\bpsi^{(k)}-\bpsi^*\| <\frac{1}{L\theta}$ and $\|F(\bpsi^{(k)})\|  \leq \frac{1}{4 L\theta^2}$, where $\theta =\|F'(\bpsi^*)^{-1}\|$ and $L$ is the Lipschitz constant in~\reff{LIP}; then a stepsize $\omega=1$ meets the backtracking condition for $n\geq k$, and $\{\bpsi^{(n)}\}$ converges to $\bpsi^*$ quadratically. 
\end{theorem}
\begin{proof}
It follows from $\|\bpsi^{(k)}-\bpsi^*\| <\frac{1}{L\theta}$, 
the Lipschitz condition on $G'$ (Lemma \ref{Lip}), and the Banach Lemma~\cite{ortega1970iterative} that 
\begin{equation}\label{InvEst}
\|F'(\bpsi^{(k)})^{-1}\|\leq \frac {\|F'(\bpsi^*)^{-1}\|}{1-\|I-F'(\bpsi^*)^{-1}F'(\bpsi^{(k)})\|}< \frac{\theta}{1-L\theta \|\bpsi^{(k)}-\bpsi^*\|}.
\end{equation}
Also, we have 
\[
\begin{split}
\|F(\bpsi^{(k)})\| &\geq \|F'(\bpsi^*)(\bpsi^{(k)}-\bpsi^*)\|-\|F(\bpsi^{(k)})-F(\bpsi^{*})-F'(\bpsi^*)(\bpsi^{(k)}-\bpsi^*)\| \\
&\geq \frac{1}{\|F'(\bpsi^*)^{-1}\|}\|\bpsi^{(k)}-\bpsi^*\|-\frac{L}{2}\|\bpsi^{(k)}-\bpsi^*\|^2 \geq  \frac{1}{2\theta}\|\bpsi^{(k)}-\bpsi^*\|.
\end{split}
\]
Therefore, we have 
\begin{equation}\label{RefinedError}
\|\bpsi^{(k)}-\bpsi^*\|\leq 2\theta\|F(\bpsi^{(k)})\| \leq\frac{1}{2L\theta}.
\end{equation} 
A combination of~\reff{RefinedError} with~\reff{InvEst} leads to  
\[
\begin{split}
\|F'(\bpsi^{(k)})^{-1} F(\bpsi^{(k)})\|^2 \leq \|F'(\bpsi^{(k)})^{-1}\|^2 \| F(\bpsi^{(k)})\|^2 \leq 4\theta^2 \cdot \frac{1}{4 L\theta^2} \| F(\bpsi^{(k)})\|  \leq \frac{1}{L}\| F(\bpsi^{(k)})\|.
\end{split}
\]
Therefore, it follows from Lemma \ref{Omega} that a stepsize $\omega=1$ satisfies the backtracking condition, i.e. $\|F(\overline{\bpsi}^{(k+1)})\| \leq  \|F(\bpsi^{(k)})\|$.

Recalling that $\overline{\bpsi}^{(k+1)}=\bpsi^{(k)}-\omega F'(\bpsi^{(k)})^{-1} F(\bpsi^{(k)})$ with $\omega=1$,
we have
\[
F'(\bpsi^{(k)})(\overline{\bpsi}^{(k+1)}-\bpsi^*) =\int_0^1  \big[F'(\bpsi^{(k)})-F'(\bpsi^*+t(\bpsi^{(k)}-\bpsi^*))\big] (\bpsi^{(k)}-\bpsi^*) \,\mathrm{d}t.
\]
Then, we have
\[
 \begin{aligned}
\|\bpsi^{(k+1)}-\bpsi^*\|&\leq \|\overline{\bpsi}^{(k+1)}-\bpsi^*\|
\leq \bigg(\|F'(\bpsi^{(k)})^{-1}\| \int_0^1(1-t) L \,\mathrm{d}t\bigg)\|\bpsi^{(k)}-\bpsi^*\|^2 \\
&=\frac{ L}{2}\|F'(\bpsi^{(k)})^{-1}\| \|\bpsi^{(k)}-\bpsi^*\|^2 \\
&\leq \frac{L\theta }{2(1-L\theta \|\bpsi^{(k)}-\bpsi^*\|)}  \|\bpsi^{(k)}-\bpsi^*\|^2 \\
&\leq L\theta \|\bpsi^{(k)}-\bpsi^*\|^2,
\end{aligned}
\]
where we have used~\reff{RefinedError} in the last step. Therefore, the sequence $\{\bpsi^{(n)} \}$ generated by the Newton iterative method with truncation is quadratically convergent. 
\end{proof}

\section{Numerical Results}
\label{s:NumRes}

Numerical simulations are performed to demonstrate the accuracy and efficiency of the proposed Newton iteration method with truncation. The calculations were done using a server with Intel(R) Xeon(R) CPU E5-2640 v4 at 2.40GHZ and 94 GB RAM. Ionic steric effects on electrostatic interactions and counterion stratification in biomolecular solvation systems are investigated by the proposed method as well. In our computations, the Yukawa potential that approximates the electrostatic potential on the boundary is used as the boundary condition~\cite{VISMPB_JCTC2014}. To improve accuracy, a modified Debye screening length due to ionic steric effects could be used in the Yukawa potential.  A more accurate treatment of boundary conditions can be derived by considering contributions from the outside of the computational box based on the linearized steric PB equation~\cite{BOSCHITSCHFENLEY_JCC07}.  The linear system in each Newton iteration step is solved with the algebraic multigrid method.  Unless otherwise specified, biomolecules are solvated in a binary monovalent ionic solution and the following parameters are used in the computations: $\ve_{\rm m}=1$, $\ve_{\rm w}=78$, $\tau = 1.5$ \AA, and $c_1^\infty= c_2^\infty= 0.1$ M. 


\subsection{Tests}
We consider that a charged macroion of radius $R$, carrying a fixed charge $Q=-5 e$ at the origin, is solvated in a binary monovalent electrolyte solution. We take $v_0=2.75^3$ \AA$^3$, $v_1=2.76^3$ \AA$^3$, $v_2=3.62^3$ \AA$^3$, $L=10$ \AA, and $R=5$ \AA~ in the computations. To test the accuracy of the numerical method, we introduce an extra source term in the equation for $\psi^r$ so that the equation has a known exact solution $\psi_{\rm ex}^r(x,y,z)=1000\exp(-\frac{x^2+y^2+z^2}{L^2})$, which is in turn used to determine boundary conditions and the extra source term. With the exact solution, we compute the relative $l^\infty$ error of numerical solutions on meshes with various resolutions and compute numerical convergence order. In computations, the Newton iterations with truncation start from a  zero initial guess and stop as the residual in $l^\infty$ norm becomes less than $tol=1$E$-6$.
\begin{table}[htbp]	
	\centering
	\caption{The relative $l^\infty$ error, convergence order, and iteration steps on meshes with various grid spacing $h$.}
	\begin{tabular}{cccc}
		\hline \hline 
		Grid spacing $h$ & Error & Order & Steps \\
		\hline 
		
		$0.4$   & 0.0348  & -- 
		
		& 6 \\

		$0.2$   & 0.0108   &  1.71
		
		& 6 \\
		
		$0.1$	 & 0.0030   &  1.87
		
		& 6 \\

		$0.05$   & 0.0008    & 1.95 
		
		& 6 \\
		\hline \hline
	\end{tabular}
	\label{OrderOfConver}
\end{table}
\begin{table}[htbp]
	\centering
	\caption{Computational time cost (in seconds) and difference between numerical solutions in $l^\infty$ norm for computations with/without the precomputing-interpolation (P-I) strategy.}
	\begin{tabular}{ccccccc}
		\hline \hline
		Grid spacing $h$  & Without P-I& With P-I & Difference\\
		\hline
		
		$0.4$  & 13.03 &  5.05  & 2.0E-10   \\
		
		$0.2$  & 107.96  & 55.65 & 2.0E-10   \\
		
		$0.1$  & 967.95 &  602.12 & 2.0E-10   \\
		
		$0.05$  & 9380.97 &  6797.67 & 2.0E-10   \\ 
		
		\hline \hline 
	\end{tabular}
	\label{CPUcomparison}
\end{table}

Table~\ref{OrderOfConver} displays the relative $l^\infty$ error and convergence order of numerical solutions on meshes with various resolutions. The error decreases robustly as the mesh refines and convergence rate gradually approaches second order for refined meshes, being consistent with our second-order discretization. As shown in the table, the Newton iterations meet the stopping criterion within 6 steps for various mesh resolutions. Such results demonstrate that the proposed Newton iteration method with truncation can solve the PB theory with steric effects robustly and efficiently.

In addition, we perform numerical simulations to assess the effect of the proposed precomputing-interpolation strategy on computational time cost and final numerical solutions. As shown in Table~\ref{CPUcomparison}, the precomputing-interpolation strategy can effectively save roughly half of the computational time cost, while only causing about $2$E$-10$ difference in final numerical solutions. It is expected that the precomputing-interpolation speed-up strategy becomes more advantageous when the number of ionic species is larger.


%
%
%

\subsection{Biomolecular Solvation} 

\begin{figure}[htbp] 
	\begin{center}
		\includegraphics[scale=0.6]{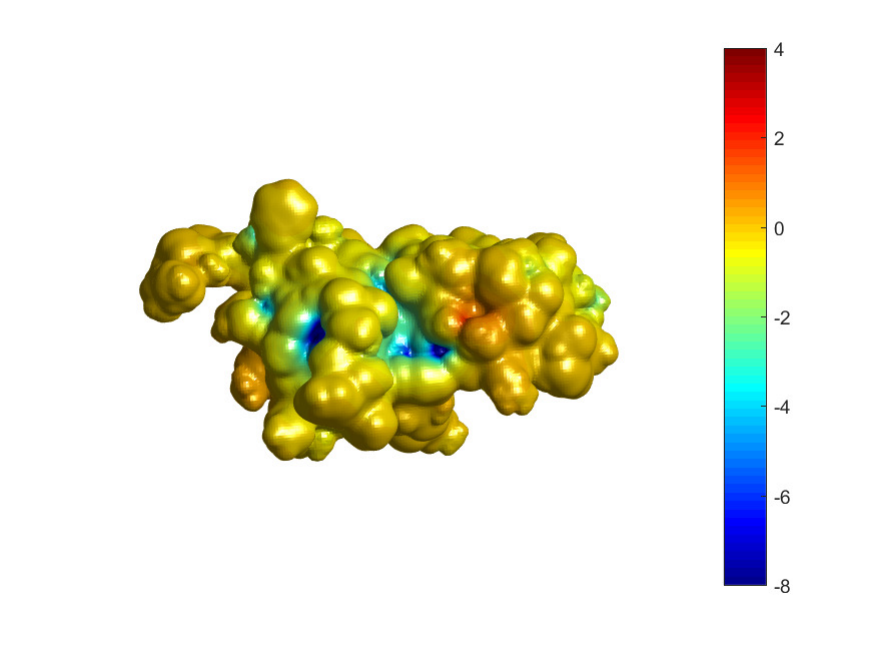}  
	\end{center} 
\vspace{-8mm}
	\caption{The difference of electrostatic potentials ($k_BT/e$) described by the steric PB and classical PB theories on the surface of the protein 1A63 computed on a mesh of grid size $200^3$.}  
	\label{f:1A63DiffPotential}
\end{figure} 
The proposed numerical method is applied to characterize the solvation of realistic biomolecules, whose atomic information, such as coordinates and partial charges, can be obtained by using the online software PDB2PQR with a Protein Data Bank (PDB) ID~\cite{DNM+:NAR:2004}. The molecular surfaces of biomolecules are calculated by a software based on a variational implicit-solvent model (VISM)~\cite{VISMPackage_JCC15}, which provides the level-set function of a molecular surface obtained by minimizing the solvation free energy.  We consider the solvation of protein molecules in a binary monovalent ionic solution with PBD ID: 1A63 and 2AID, which have been studied with a linearized PB theory without steric effects~\cite{ZhongRenTsai_JCP2018, BajajChenRand_SISC2011, GK:JCP:2013}. We take $v_0=2.75^3$ \AA$^3$, $v_1=5.51^3$ \AA$^3$, and $v_2=6.37^3$ \AA$^3$ in numerical simulations. 

For the protein 1A63, we apply the Newton iteration method with truncation to solve both the steric PB theory (sPB) and classical PB theory (cPB)  on a mesh of grid size $200^3$. The iterations take the Yukawa potential as the initial guess. The iterations for the cPB case converge in $9$ steps taking 795.92 seconds with the stepsize $\omega$ determined by the backtracking strategy, while that for the sPB case converge robustly in 4 steps taking 347.85 seconds with $\omega=1$ throughout the iterations. The computational time cost shows that the proposed Newton iteration method with truncation for the sPB theory is as efficient as that for the cPB theory in each iteration. Fig.~\ref{f:1A63DiffPotential} depicts the difference of the electrostatic potentials described by the sPB and cPB theories on the molecular surface. It is found that the electrostatic potential is stronger on the surface when it is described by the sPB theory, especially inside the groove. This is ascribed to the fact that less counterions distribute next to the charged molecules due to steric hindrance, agreeing with the existing studies~\cite{BAO_PRL97, LiLiuXuZhou_Nonlinearity13}. In contrast, the electrostatic potential predicted by the cPB theory is much screened by more closely attracted counterions and therefore becomes relatively weak on the molecule surface.

\begin{table}[htbp]
	\centering
	\caption{The reaction field energy ($k_BT$), and range of electrostatic potential ($k_BT/e$) and maximum concentrations (M) on the surface of the protein 1A63, calculated by the classical PB (cPB) and steric PB (sPB) theories. }  
	\begin{tabular}{ccccccc}
		\hline \hline
		&\multicolumn{2}{c}{$c_1^{\infty}=c_2^{\infty}=0.01 M$} &\multicolumn{2}{c}{ $c_1^{\infty}=c_2^{\infty}=0.05 M$} &\multicolumn{2}{c}{$c_1^{\infty}=c_2^{\infty}=0.1 M$ }  \\ 
		\cline{2-3} \cline{4-5} \cline{6-7}
		&cPB &  sPB  &cPB &  sPB  &cPB &  sPB \\
		\hline 
		Reaction field energy & -10287.15   &  -10274.48    & -10296.42   & -10278.97   & -10301.69   & -10281.69    \\
		
		Potential max.  & 10.19   &  11.24   & 9.12  &11.07  &8.60  & 10.97  \\ 
		
		Potential min.   & -15.98  &-37.05   &-14.50 &-36.76    &-13.87  &-36.65  \\  
		
		Cation $c_1$ max.  & 9.34E+4 &9.68  &1.22E+5 &9.72  &1.39E+5   &9.74    \\ 
		
		Anion $c_2$ max.  & 183.14  & 2.39 
		& 356.79  &  2.82 & 444.04   & 3.01   \\ 
		\hline \hline 
	\end{tabular}
	\label{T:1A63}
\end{table}
To further unravel the difference between steric PB and classical PB theories, we perform numerical simulations with various bulk concentrations and quantitatively compare the electrostatic potential and maximum concentrations on the surface, as well as the reaction field energy that is defined by $\frac{1}{2}\sum_{i=1}^N Q_i\psi^r(x_i)$. As listed in Table~\ref{T:1A63}, the reaction field energies predicted by the cPB theory are more negative than that of the sPB theory. This can be explained by the fact that electrostatic interactions between ions and biomolecules are weaker with the presence of steric hindrance.  As the bulk concentration increases, the reaction field energies become more negative and the discrepancy between cPB and sPB enhances as well. In contrast to a mild difference shown in Fig.~\ref{f:1A63DiffPotential}, data listed in the table demonstrate that the electrostatic potential on the surface actually has a large difference between cPB and sPB theories. As the bulk concentration increases, the range of the potential narrows in that more screening effect comes from mobile ions. In addition, the maximum ionic concentrations on the surface illustrate that the cPB theory predicts unphysically high counterion concentrations, evidencing the necessity of inclusion of ionic steric effects in mean-field electrostatic modeling.


\begin{figure}[htbp]
	\centering
	\includegraphics[scale=0.6]{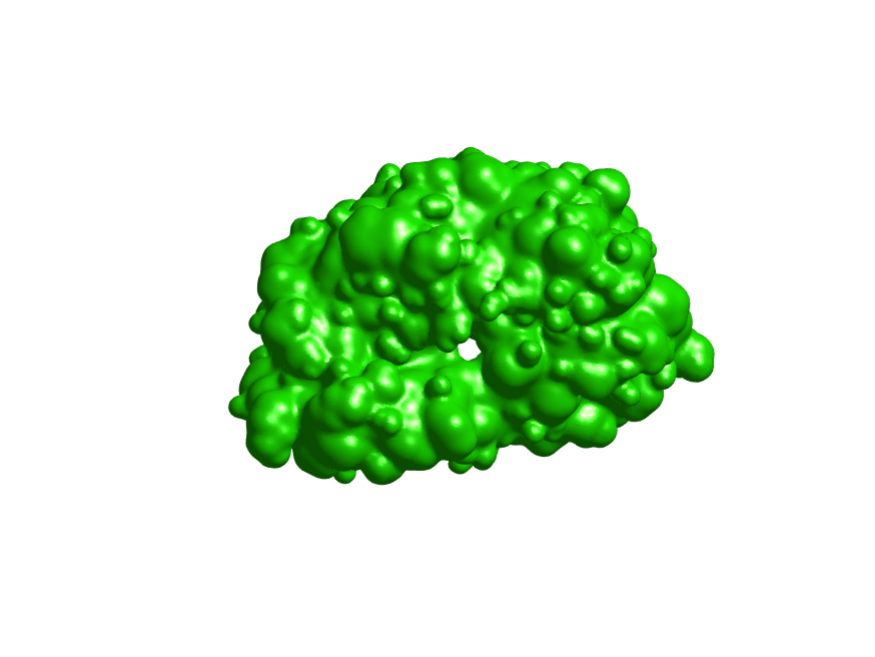}\hspace{5mm}
	\includegraphics[scale=0.6]{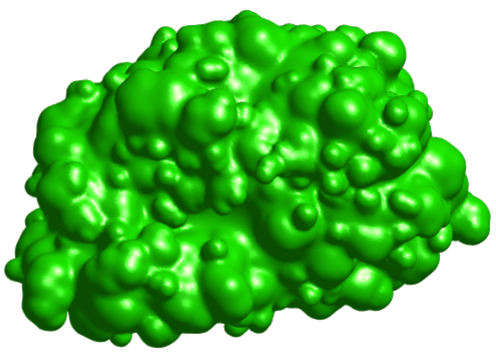}
	\caption{Molecule surfaces for two solvation states of the protein 2AID. Left: A wet state with a cavity in the center. Right: A dry state without a cavity.} 
	\label{f:2AIDSurf}
\end{figure}
\begin{figure}[htbp]
	\centering
	\includegraphics[scale=0.5]{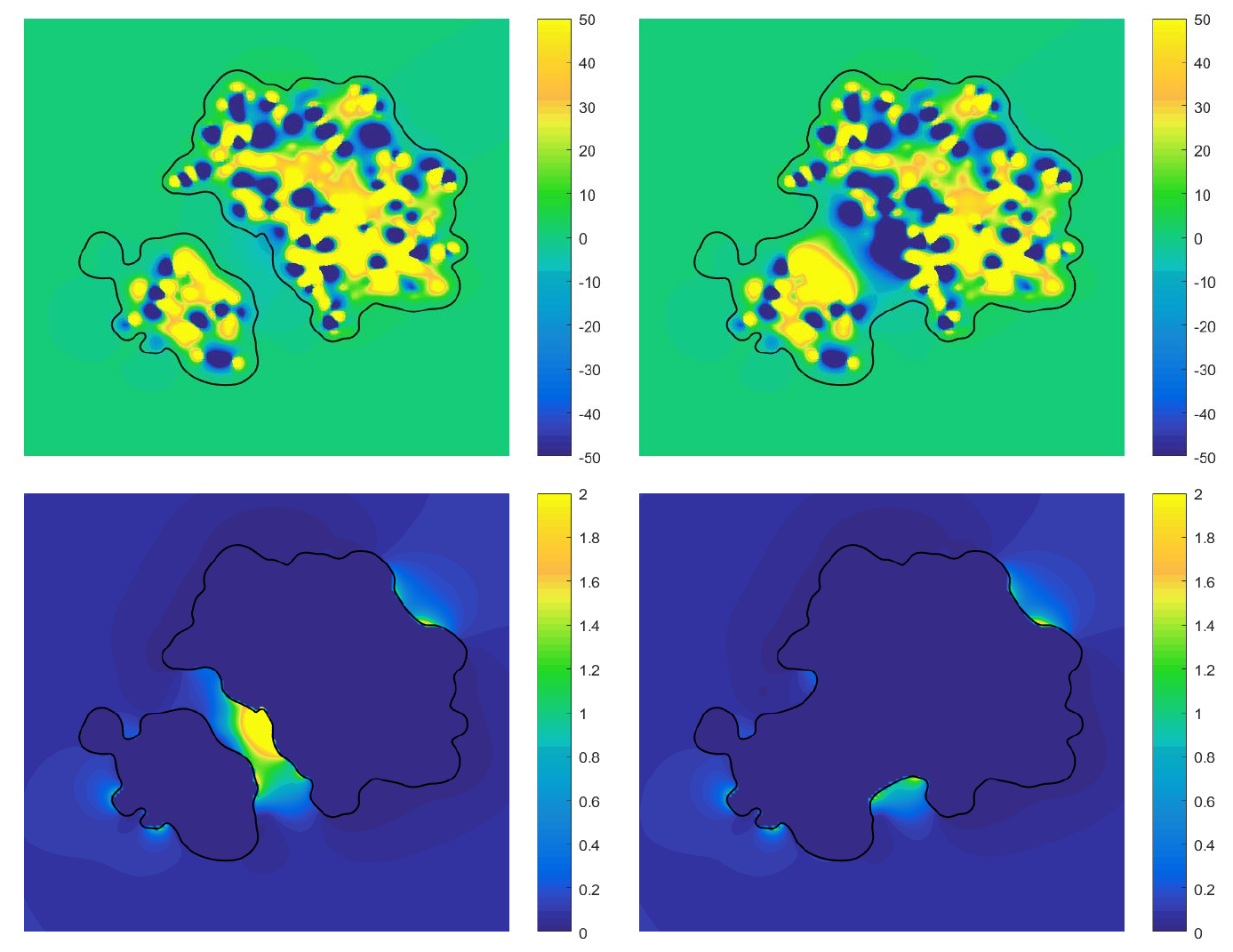}  
	\caption{The electrostatic potential ($k_BT/e$) (Upper row) and cation concentrations (M) (Lower row) on a cross-section plane of the protein 2AID in a wet state (Left column) and a dry state (Right column). Black curves depict the molecular surfaces on the cross-section plane. } 
	\label{f:2AID}
\end{figure}
We also consider the solvation of the protein 2AID, which is a non-peptide inhibitor complexed with the HIV-1 protease~\cite{ZhongRenTsai_JCP2018}. The molecular surfaces are calculated by a software based on a variational implicit-solvent model (VISM), in which the solvation free energy is minimized with respect to all possible surfaces~\cite{VISMPackage_JCC15}. Starting from different initial guesses, e.g., a tight and a loose initial wrap, the level-set relaxation of a non-convex solvation free-energy functional in VISM may lead to different molecular surfaces, corresponding to multiple solvation states of biomolecules. Wetting and dewetting transitions between different solvation states are often observed in molecular dynamics simulations~\cite{Zhou_PNAS19, ZhouZhangChengLi_JCP21}. For this case, the VISM calculations with different initial guesses predict two solvation states: a wet state where the central cavity is hydrated, and a dry state where the central cavity is dehydrated.  The left plot of Fig.~\ref{f:2AIDSurf} illustrates the wet state with a wet cavity in the center, but water molecules cannot penetrate into the cavity in the dry state as shown in the right plot. It is of particular  interest to investigate the impact of different molecule surfaces on the distributions of the electrostatic potential and counterions. 

Numerical simulations on the protein 2AID are performed on a mesh with grid spacing about $0.37$ \AA. The Newton iterations converge robustly and efficiently in 4 steps. Fig.~\ref{f:2AID} presents plots of the electrostatic potential and cation concentrations on a cross-section plane of the protein in wet and dry states. It is of interest to observe that the molecular surface in the wet state has two separate parts on the cross-section plane. Clearly, one can find that the electrostatic potential inside the cavity is much weaker in the wet state, because the screening effect is much stronger with a higher dielectric coefficient ($\ve_{\rm w} $~vs.~$ \ve_{\rm m}$) and attracted counterions. In contrast, there is a strong electric field in the same region for the dry state. Furthermore, it is depicted that, due to electrostatic interactions, cations penetrate into the cavity in the wet state and form electric double layers next to charged atoms in the protein.  
\begin{figure}[htbp]
	\centering
	\includegraphics[scale=0.55]{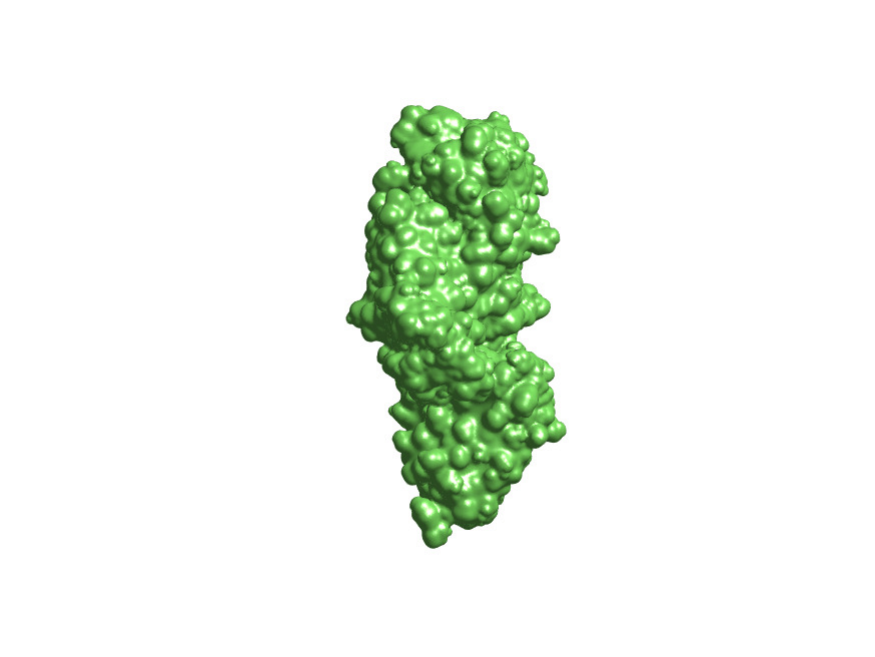}\hspace{-15mm}
	\includegraphics[scale=0.55]{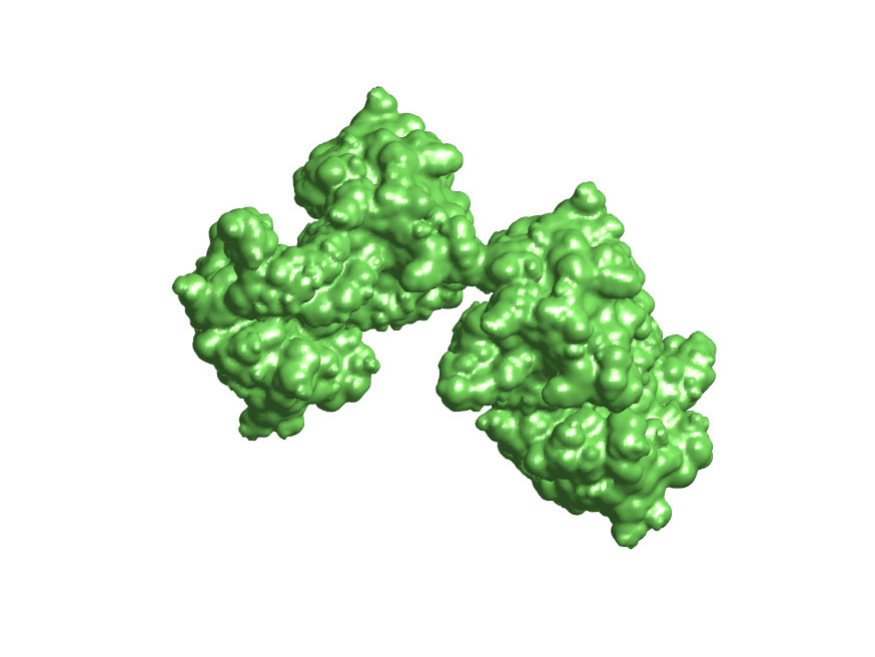}
	\vspace{-8mm}
	\caption{Molecule surfaces of the G-actin 6RSW and B-DNA 7OGS on a mesh of grid size $200^3$.} 
	\label{f:MoleSurf}
\end{figure}

\begin{table}[htbp]
	\caption{The computational time cost (in seconds), iteration steps, total net charge ($e$), atom number, and computational box size ($\AA^3$) of various molecules on a mesh of grid size $200^3$.}
	\begin{center}
		\begin{tabular}{c| c c c c c c}
			\hline
			\hline
			&    & CPU time & Steps &  Net charge & Atoms & Box size \\
			\hline
			Portein & 1A63 & 347.85 & 4 &  -1.00 & 2065 & $37.73^3$ \\
			Portein & 2AID  &  338.08  & 4 & 8.00  & 3445 & $37.39^3$ \\
			\hline
			Virus & 1F15  &  370.31  & 4 & 12.00  & 8494 & $46.23^3$ \\
			G-actin & 6RSW  &  359.59  & 4 & -14.00  & 10811 & $53.70^3$ \\
			G-actin & 6GVC  &  326.84  & 4 & -18.34  & 37447 & $56.65^3$ \\
			B DNA & 7OGS  &  350.89  & 4 & -68.00  & 10078 & $65.30^3$ \\
			B DNA & 7Q0N  &  342.34  & 4 & -100.00 & 15551 & $67.12^3$ \\
			\hline
			\hline
		\end{tabular}
	\end{center}
	\label{SeveralMoleCPU}
\end{table}

To further assess the performance of our numerical approach,  we perform additional simulations on several biomolecules that have large sizes, different topology, and high net charges. For instance, as shown in Fig.~\ref{f:MoleSurf},  we study the G-actin protein 6RSW that has a long structure, and B-DNA 7OGS that has two connected regions. Table~\ref{SeveralMoleCPU} demonstrates that the proposed numerical method is still efficient for biomolecular systems with high net charges and large sizes, compared with previously studied proteins 1A63 and 2AID. For instance, the protein 6GVC has $37447$ atoms and B-DNA 7Q0N has total net charge $-100\,e$. The Newton iterations for these systems converge within $4$ steps taking $5$-$6$ minutes.

\subsection{Counterion Stratification} 
\begin{figure}[htbp] 
	\begin{center}
		\includegraphics[scale=0.9]{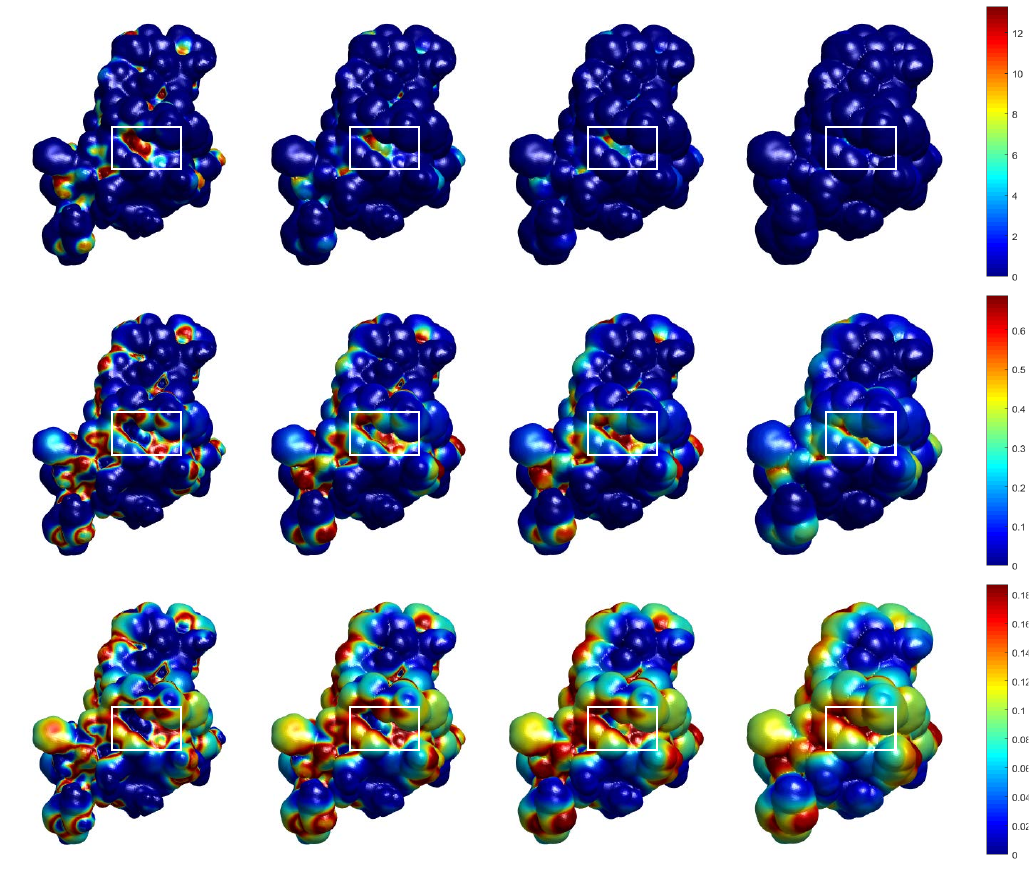}  
	\end{center}
	\caption{Upper ($+3$), middle ($+2$), and lower ($+1$) rows present concentrations next to the protein 1A63 for counterions with ratios $\alpha_{+3}:\alpha_{+2}:\alpha_{+1}=3:2:1$. The first to the fourth column correspond to iso-surfaces with level-set values $0.001$, $0.5$, $0.8$, and $1.5$ \AA, respectively.} 
	\label{f:1A63ConStrst1}
\end{figure}  
Counterions stratify near highly charged surfaces, resulting from the competition between the entropy effect and electrostatic interactions. An ionic valence-to-volume ratio parameter, $\alpha_i=\frac{|z_i|}{v_i}$, was first proposed to describe the order of counterion stratification: the counterion species with larger valence-to-volume ratio can distribute closer to the charged surface~\cite{ZhouWangLi_PRE11}. The role of the parameter later was further confirmed by asymptotic analysis~\cite{LiLiuXuZhou_Nonlinearity13} and Monte Carlo simulations~\cite{WenZhouXuLi_MC_PRE12}. However, the numerical simulations and analysis in these works were performed in a simple spherical geometry. It is of significance to consider counterion stratification next to realistic biomolecular surfaces. 
\begin{figure}[htbp] 
	\begin{center}
		\includegraphics[scale=0.9]{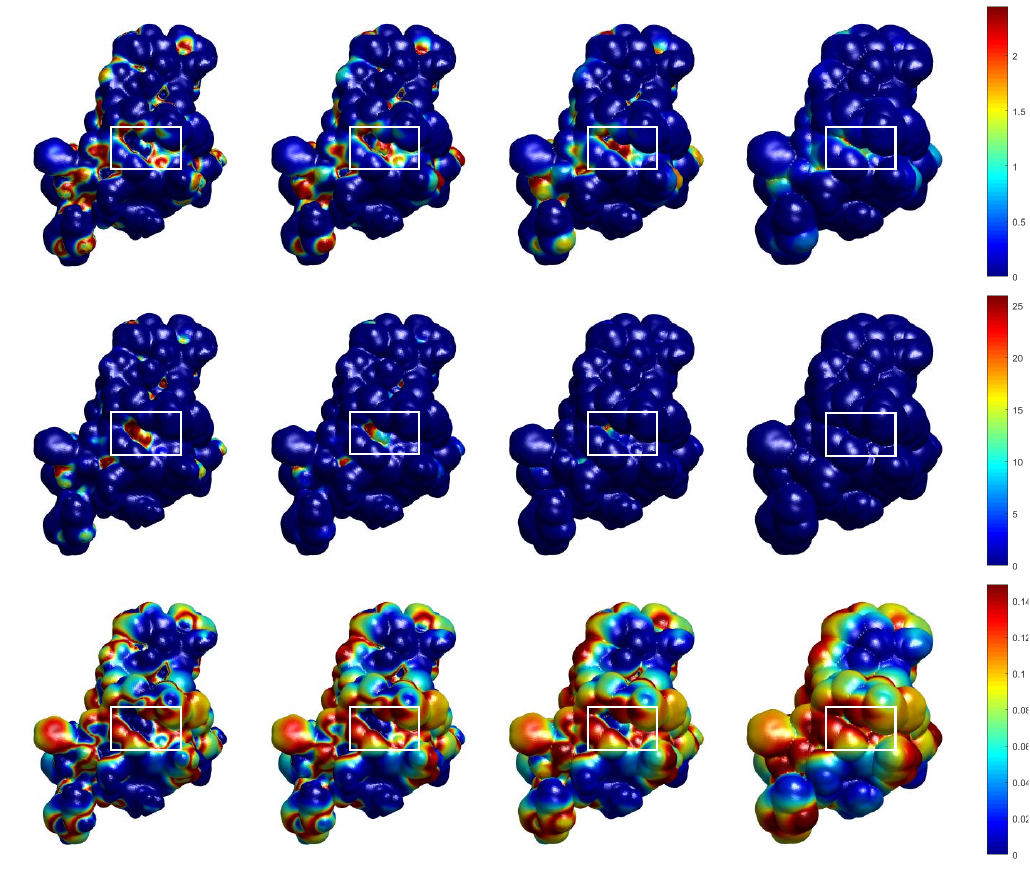}  
	\end{center}
	\caption{Upper ($+3$), middle ($+2$), and lower ($+1$) rows present concentrations next to the protein 1A63 for counterions with ratios $\alpha_{+3}:\alpha_{+2}:\alpha_{+1}=3:6.475:1$. The first to the fourth column correspond to iso-surfaces with level-set values $0.001$, $0.3$, $0.7$, $1.5$ \AA, respectively.}  
	\label{f:1A63ConStrst2}
\end{figure} 
In computations, we consider an ionic solution that consists of four species with $z_1=+1$, $z_2=+2$, $z_3=+3$, $z_4=-1$, $c_1^\infty=c_2^\infty=c_3^\infty=0.1$ M, and $c_4^\infty=0.6$ M.  To verify the role of $\alpha_i$ in realistic biomolecular systems, we take two groups of ionic volumes:
	\[
	\begin{aligned}
	&\mbox{Group I: } v_0=3^3 {\rm \AA}^3, v_1=v_2=v_3=5^3 {\rm \AA}^3,  v_4=6^3 {\rm \AA}^3, \mbox{ and } \alpha_{+3}:\alpha_{+2}:\alpha_{+1}=3:2:1; \\
	&\mbox{Group II: } v_0=3^3 {\rm \AA}^3, v_1=v_3=v_4=6^3 {\rm \AA}^3, v_2=4^3 {\rm \AA}^3, \mbox{ and } \alpha_{+3}:\alpha_{+2}:\alpha_{+1}=3:6.475:1.
	\end{aligned}
	\]
Fig.~\ref{f:1A63ConStrst1} presents the counterion concentration next to the protein 1A63 with ionic volumes given in Group I. The ionic concentrations are plotted on iso-surfaces with increasing level-set values in columns. Note that higher level-set values mean farther distances to the molecule surface, which is the iso-surface with the zero level-set value. It is of interest to observe that  three species of counterions stratify in certain highly charged regions, e.g., the regions highlighted by boxes. From the figure, one observes that the trivalent ions distribute closest to the negatively charged surface and the concentration quickly decreases as the distance increases. The peak value is close to its saturation concentration, i.e., $\frac{1}{v_{3}}$.  For divalent ions, the concentration increases quickly first,  reaches its peak at the distance about $0.8$ \AA, and decreases for larger distances. For monovalent ions, the concentration increases from almost zero to its saturation concentration at the distance about $1.5$ \AA. Overall, one can find that three species of counterions stratify clearly into three layers with the order prescribed by the value of $\alpha_i$ ($\alpha_{+3}:\alpha_{+2}:\alpha_{+1}=3:2:1$). 

The role of $\alpha_i$ is further confirmed by the simulations with ionic volumes given by Group II, which has $\alpha_{+3}:\alpha_{+2}:\alpha_{+1}=3:6.475:1$. From Fig.~\ref{f:1A63ConStrst2}, one can see that the divalent ions, which have the largest $\alpha_i$ value, take the lead to be attracted to the negatively charged regions. As the divalent concentration decreases, the concentration of monovalent ions that have the second largest $\alpha_i$ value increases to its peak value at the distance about $0.7$ \AA. With the smallest $\alpha_i$ value, the monovalent ions distribute in the third layer and increases to its peak value after the concentration of divalent and trivalent ions both decrease. Such a layering order again confirms the role of the  parameter. Therefore, the valence-to-volume ratio parameter still works well for realistic, complicated geometry and can have great potential in predicting counterion layering structures in realistic biomolecular systems under physiological conditions.

\section{Conclusions}
\label{s:Conclusions}
Local steric Poisson--Boltzmann (PB) theories have been widely applied to describe ionic size effects in ionic solutions. This work has proposed a fast Newton iteration method with truncation to solve local steric PB theories. To present the method, we have focused on one local steric PB theory that is derived from a lattice-gas theory.  A crucial idea that generalized Boltzmann distributions are numerically available accounts for the advantages achieved by the proposed method in efficiency and memory saving. The existence, uniqueness, boundness, and smoothness of the generalized Boltzmann distributions have been rigorously established. Also, the existence and uniqueness of the solution to the nonlinear system discretized from the generalized PB equation have been established by showing that it is a unique minimizer of a constructed convex energy. By the extremum principle, the upper and lower bounds for the solution have been obtained with the boundness on ionic concentrations. Detailed analysis has revealed that the truncation step in Newton iterations further decreases the energy and error. To further speed-up computations, we have also proposed a novel precomputing-interpolation strategy, which is applicable to other local steric PB theories and makes the proposed methods for solving steric PB theories as efficient as for solving the classical PB theory.  Further analysis has proved local quadratic convergence rate for the proposed Newton iteration method with truncation. 

Numerical simulations have demonstrated the effectiveness and high efficiency of the proposed method. Applications to realistic biomolecular solvation systems have illustrated that the proposed numerical method for the steric PB theories can effectively capture steric effects in biomolecular solvation and have great potential in predicting counterion stratification next to large-scale charged biomolecules in physiological conditions. Finally, we highlight that the proposed numerical method for local steric PB theories can be readily incorporated in the well-known classical PB solvers to consider ionic steric effects, e.g., the APBS~\cite{BSJ+:PNASUSA:2001, APBS_Baker18}, DelPhi~\cite{DelPhiAlexov_CICP2013, LiJiaPengAlexov_BioinfM2017},  MIBPB~\cite{ZhouFeigWei_JCC2008, ChenGengWei_JCC2011}, and AFMPB~\cite{LuChengHuangMcCammon_CPC13}.

\section*{Acknowlegment}
We would like to thank the anonymous reviewers for their comments which have led to an improvement of  this paper.  M. Chen was supported by National Natural Science Foundation of China through Grant No. 11801513 and Fundamental Research Funds of Zhejiang Sci-Tech University through Grant No. 2021Q053. W. Dou and S. Zhou were partially supported by the National Natural Science Foundation of China 12171319 and Shanghai Science and Technology Commission (21JC1403700).
}
\bibliography{Charges,hydrophob}
\bibliographystyle{plain}
\end{document}